\documentclass[siamltex]{elsarticle}
\usepackage{amsmath}
\usepackage{amssymb}
\usepackage{amsbsy}
\usepackage{epsfig}
\usepackage{graphicx}
\usepackage{colordvi}
\usepackage{color}
\usepackage{graphics}
\usepackage{graphicx}
\usepackage[left=1in, top=1in, right=1in, bottom=1in]{geometry}
\usepackage{amsfonts}
\usepackage{verbatim}
\usepackage{tabularx}
\usepackage{tikz}
\usetikzlibrary{arrows}
\usetikzlibrary{positioning}
\usepackage{lscape}
\usepackage{subfigure}
\usepackage{wrapfig}
\usepackage{floatrow}
\usepackage{esvect}  %arrow of vector
\usepackage[utf8]{inputenc}  % for gather
\numberwithin{equation}{section}

\usepackage[english]{babel}
\usepackage[utf8]{inputenc}
\usepackage{amsmath}

\usepackage{graphicx}
\usepackage[colorinlistoftodos]{todonotes}
\usepackage{float}

\newtheorem{theorem}{Theorem}[section]
\newtheorem{proof}{Proof}[section]

\newtheorem{lemma}[theorem]{Lemma}
\newtheorem{proposition}[theorem]{Proposition}
\newtheorem{definition}[theorem]{Definition}

\newcommand{\om}{\Omega}

\begin{document}

\title{Defect-Deferred Correction Method Based on a Subgrid Artificial Viscosity Modeling}

\author{Mustafa Aggul}
\ead{mustafaaggul@hacettepe.edu.tr}

\address{Department of Mathematics, Hacettepe University, 06800, Ankara, Turkey.}
\vspace{0.2in}

\date{\today}

\begin{abstract}  

An alternative first step approximation based on subgrid artificial viscosity modeling (SAV) is proposed for defect-deferred correction method (DDC) for incompresible Navier-Stokes equation at high Reynolds number. This new approach not only preserves all qualifications of the conventional artificial viscosity (AV) based DDC, such as unconditional stability, high order of accuracy and so on, it has also shown its superiority over choosing AV approximation in the predictor step. Both theory and computational results presented in this paper illustrate that this alternative approach indeed increases the efficiency of the DDC method.

\end{abstract}


\maketitle

%%%%%%%%%%%%%%%%%%%%%%%%%%%%%%%%%%%%%%%%%%%%%%%%%%%%%%%%%%%%%%%%%%%%%%%%%%%%%%%%%%%%%%%%%%%%%%%%%%%%

\section{Introduction}

In this report, for the pair of unknown velocity $\mathbf{u}$ and pressure $\mathbf{p}$, we consider the incompressible Navier-Stokes equation (NSE) \ref{NSE} at high Reynolds number ($Re^{-1} \propto \nu$). According to Kolmogorov's K41 Theory \cite{K41}, as Reynolds number increases required computational cost raises prohibitively high. Attempting to solve the problem directly with an affordable computational cost(on a much coarser mesh than required) usually causes related linear systems to converge too slowly, or even if they converge within a reasonable time frame, their results are far from being realistic. 

\begin{eqnarray}\label{NSE}
\mathbf{u}_{t} -\nu \Delta \mathbf{u} +\mathbf{u}\cdot \nabla \mathbf{u}+\nabla \mathbf{p}=f,\notag\\
\nabla \cdot \mathbf{u}=0.
\end{eqnarray}

Various techniques including the defect correction have been introduced to mitigate this issue, see \cite{LLP02,ELM00,L08}. Also a recent work combining this correction approach with a deferred correction (see e.g. \cite{DGR00,KG02,M03,M04}) for an increased temporal accuracy have been proposed in \cite{AL16}. Methods on both of these papers are based on a predictor-corrector scheme: As a predictor step, an approximation is found by a computationally very attractive artificial viscosity(AV) approximation with a backward-Euler time discretization, and as for the corrector step, the affect of the AV is subtracted via previously found predictor step approximation. In particular, see the following scheme for the artificial viscosity based defect-deferred correction method(AV-DDC) presented in \cite{AL16}:

\begin{equation}\label{AVapp}
\begin{split}
(\frac{u_1^{h,n+1} - u_1^{h,n}}{k}, v^h) + (\nu + h)(\nabla u_1^{h,n+1},\nabla v^h) + b^{\ast }(u_1^{h,n+1}, u_1^{h,n+1}, v^h) \\- (p_1^{h,n+1},\nabla \cdot v^h) = (f(t_{n+1}), v^h)
\end{split}
\end{equation}

\begin{equation}\label{CSapp}
\begin{split}
(\frac{u_2^{h,n+1}-u_2^{h,n}}{k},v^h)+(\nu + h)(\nabla u_2^{h,n+1},\nabla v^h) + b^{\ast }(u_2^{h,n+1}, u_2^{h,n+1}, v^h)
\\
- (p_2^{h,n+1},\nabla \cdot v^h)
=(\frac{f(t_{n+1})+f(t_{n})}{2},v^h)
+\frac{\nu}{2}k(\nabla (\frac{u_1^{h,n+1}-u_1^{h,n}}{k}),v^h)\\
+\frac{1}{2}b^{\ast }(u_1^{h,n+1}, u_1^{h,n+1}, v^h)
-\frac{1}{2}b^{\ast }(u_1^{h,n}, u_1^{h,n}, v^h)
+h(\nabla u_1^{h,n+1},\nabla v^h),
\end{split}
\end{equation}

where $b^{\ast }(\cdot,\cdot,\cdot)$ is the explicitly
skew-symmetrized trilinear form, defined below.

AV-DDC beside being an efficient method, it has been successfully applied to various problems including two-domain convection-dominated convection diffusion problem and nonlinearly-coupled fluid-fluid interaction, see \cite{EL16}, \cite{ACEL18}. In all of these papers, it has been shown to be an unconditionally-stable, high-accuracy regularization technique (a second order in time and space). It is also parallelizable for a faster result since only data transfer required for the correction step is the AV solutions on the current and previous time steps. Therefore, one can easily run the scheme in parallel as long as AV approximation marches only two time steps earlier than the correction steps.

On the other hand, the accuracy of the correction step approximation is strongly dependent on the accuracy of the predictor step, in general. Especially, for AV-DDC methods, accuracy of the correction step is lifted by an order of 1 due to the multiplication of $h$ in the laplacian of the AV approximation, see the last term of the equation \ref{CSapp}. Also AV approximation is known to be too dissipative (in all scales) so that it cannot capture turbulent characteristics of the flow and results in a fully-laminar flow, e.g. see \cite{AL16}. Therefore, replacing the predictor AV step with a less dissipative and high-accuracy approximation fosters the overall accuracy of the first step approximation and, in consequence, correction step approximation produces better solutions. In this report, AV approximation in the first step will be replaced with a projection-based subgrid artificial viscosity method (SAV) to further increase the accuracy of the correction step approximation, see e.g. \cite{JK05,HMJ00,C01,JKL06,KLR06} for SAV and its inspiration source variational multiscale methods (VMS).

Hence, the replacement of \ref{AVapp} with \ref{VMSapp} for the predictor step is proposed, and the new defect-deferred correction based on SAV is abbreviated to SAV-DDC. In contrast to commonly used coupled (implicit) form of SAV in the literature this replacement decouples the projection step from the NSE for computational efficiency. Although this decoupling comes with an extra $O(\Delta t)$ error, decoupled SAV still meets with our expectations from the predictor step approximation since it only has to be first order of accuracy as in AV approximation. One can also employ the implicit form of SAV for possibly better accuracies.

\begin{eqnarray}\label{VMSapp}
(\frac{u_1^{h,n+1} - u_1^{h,n}}{k}, v^h) + (\nu + h)(\nabla u_1^{h,n+1},\nabla v^h) + b^{\ast }(u_1^{h,n+1}, u_1^{h,n+1}, v^h) \nonumber\\
- (p_1^{h,n+1},\nabla \cdot v^h) = (f(t_{n+1}), v^h) +h(\mathbb{G}_1^{\mathbb{H},n},\nabla v^{h}),\\
(\mathbb{G}_1^{\mathbb{H},n}-\nabla u_1^{h,n}, \mathbb{L}^H)=0.\label{projection} 
\end{eqnarray}

The equation \ref{projection} means that $\mathbb{G}_1^{\mathbb{H},n}$ is the projection of $\nabla u_1^{h,n}$ on a coarse mesh. Consider 

$$(\nu + h)(\nabla u_1^{h,n+1},\nabla v^h) - h(\mathbb{G}_1^{\mathbb{H},n},\nabla v^h) = \nu (\nabla u_1^{h,n+1},\nabla v^h) + h(\nabla u_1^{h,n+1}-\mathbb{G}_1^{\mathbb{H},n},\nabla v^h).$$ 

Roughly ignoring the difference in the corresponding time levels, the last term corresponds to the gradient of the small scales that would disappear upon the projection onto the given coarse mesh. Therefore, we infer that the dissipative affect of the artificial viscosity is only introduced on small scales and it acts solely indirectly on large scales, see \cite{JK05} for details. This distinction results in resolving large eddies (the ultimate goal of most practitioners) with a much higher accuracy with SAV than AV approximation, which acts on all scales regardless of their sizes.

The paper is organized as follows: Section $\ref{prelim}$ introduces the necessary notation and preliminaries; then Section 3 follows with the accuracy and stability of the first step SAV approximation. The main theoretical results of the proposed replacement, SAV-DDC appears in Section 4, where stability and increased accuracy (both time and space) of the correction step is studied. Computational comparison tests are presented in Section 5.

%%%%%%%%%%%%%%%%%%%%%%%%%%%%%%%%%%%%

\section{Mathematical Preliminaries and Notations}\label{prelim}
Throughout this paper, the norm $||.||$ denotes the usual $L^2(\Omega)$ norm of scalars, vectors and tensors, induced by the usual $L^2$ inner-product, denoted by $(\cdot,\cdot)$. The space in which velocity sought(at time $t$) is

\[X=H_0^1(\Omega)^d=\{v \in L^2(\Omega)^d:\nabla v \in L^2(\Omega)^{dxd} \mbox{ and } v=0 \mbox{ on } \partial \Omega \}. \]
with the norm $||v||_X=||\nabla v||.$ The space dual to $X$ is equipped with the norm
\[||f||_{-1}=\sup_{v \in X} \frac{(f,v)}{||\nabla v||}. \]

The space that pressure (at time $t$) belongs to is

\[Q=L_0^2(\Omega)=\{q \in L^2(\Omega):\int_{\Omega} q(x)dx=0 \}. \]

Introduce the space of weakly divergence-free functions

\[X \supset V=\{v \in X: (\nabla \cdot v,q)=0, \forall q\in Q  \}.\]

For measurable $v:[0,T] \rightarrow X $, we define
\[||v||_{L^p(0,T;X)}=(\int_{0}^T ||v||_X^Pdt)^\frac{1}{p}, 1 \leq p < \infty, \]
and
\[||v||_{L^\infty(0,T;X)}=ess \sup_{0\leq t \leq T} ||v(t)||_X. \]
Define the trilinear form on $X\times X \times X$
\[b(u,v,w)=\int_{\Omega}u\cdot \nabla v \cdot w dx.\]

The following lemma is also necessary for the analysis.
\begin{lemma}\label{prelim1}
There exist finite constant $M=M(d)$ and $N=N(d)$ s.t. $M \geq N$ and
\[M=\sup_{u,v,w \in X} \frac{b(u,v,w)}{||u||||v||||w||} < \infty, N=\sup_{u,v,w \in V} \frac{b(u,v,w)}{||u||||v||||w||} < \infty. \]
\end{lemma}

The proof can be found, for example, in ~\cite{GR79}. The
corresponding constants $M^h$ and $N^h$ are defined by replacing $X$
by the finite element space $X^h \subset X$ and $V$ by $V^h \subset
X$, which will be defined below. Note that $M \ge \max(M^h, N, N^h)$
and that as $h \rightarrow 0$, $N^h \rightarrow N$ and $M^h
\rightarrow M$ (see ~\cite{GR79}).

Throughout the paper, we shall assume that the velocity-pressure
finite element spaces $X^h \subset X$ and $Q^h \subset Q$ are
conforming, have typical approximation properties of finite element
spaces commonly in use, and satisfy the discrete inf-sup, or
$LBB^h$, condition
\begin{eqnarray} \label{inf_sup}
\inf_{q^h \in Q^h} \sup_{v^h \in X^h} \frac{(q^h, \nabla \cdot
v^h)}{\|\nabla v^h\|\|q^h\|} \ge \beta^h > 0,
\end{eqnarray}
where $\beta^h$ is bounded away from zero uniformly in $h$. Examples
of such spaces can be found in ~\cite{GR79}. We shall consider $X^h
\subset X$, $Q^h \subset Q$ to be spaces of continuous piecewise
polynomials of degree $m$ and $m-1$, respectively, with $m \ge 2$.
The case of $m=1$ is not considered, because the optimal error
estimate (of the order $h$) is obtained after the first step of the
method - and therefore the DCM in this case is reduced to the
artificial viscosity approach.

The space of discretely divergence-free functions is defined as
follows
\begin{eqnarray*}
V^h = \{ v^h \in X^h : (q^h, \nabla \cdot v^h) = 0, \forall q^h \in
Q^h \}.
\end{eqnarray*}
%We will also use the $H\ddot older$'s and Young's inequalities
%frequently: for any $\epsilon$, $0 < \epsilon < \infty$ and
%$\frac{1}{p} + \frac{1}{q} = 1$, $1 \le p,q \le \infty$:
%\begin{eqnarray*}
%(u,v) \le \|u\|_{L^p}\|v\|_{L^q} \mbox{, and } (u,v) \le \frac{\epsilon}{p}\|u\|_{L^p}^p + \frac{\epsilon^{-\frac{q}{p}}}{q}\|v\|_{L^q}^q
%\end{eqnarray*}

In the analysis we use the properties of the following Modified
Stokes Projection
\begin{definition} [Modified Stokes Projection] \label{Modified Stokes Projection}
Define the Stokes projection operator $P_S$: $(X,Q) \rightarrow
(X^h, Q^h)$, $P_S(u,p) = (\tilde u, \tilde p)$, satisfying
\begin{eqnarray} \label{jan1}
(\nu + h)(\nabla(u - \tilde u), \nabla v^h) - (p - \tilde p,
\nabla \cdot v^h) = 0,\\
\nonumber (\nabla \cdot (u - \tilde u), q^h)=0,
\end{eqnarray}
for any $v^h \in V^h, q^h \in Q^h$.
\end{definition}

In $(V^h, Q^h)$ this formulation reads: given $(u,p) \in (X,Q)$,
find $\tilde u \in V^h$ satisfying
\begin{eqnarray} \label{jan5}
(\nu + h)(\nabla(u - \tilde u), \nabla v^h) - (p - q^h, \nabla
\cdot v^h) = 0,
\end{eqnarray}
for any $v^h \in V^h, q^h \in Q^h$.

Define the explicitly skew-symmetrized trilinear form
\begin{equation*}
b^{\ast }(u,v,w):=\frac{1}{2}(u \cdot \nabla v,w)-\frac{1}{2}(u
\cdot \nabla w,v).
\end{equation*}
The following estimate is easy to prove (see, e.g., ~\cite{GR79}):
there exists a constant $C=C(\Omega)$ such that
\begin{equation} \label{nonlin_usual}
|b^{\ast }(u,v,w)| \le C(\Omega)\|\nabla u\|\|\nabla v\|\|\nabla w\|.
\end{equation}

The proofs will require the sharper bound on the nonlinearity. This
upper bound is improvable in $R^{2}$.
\begin{lemma} [The sharper bound on the nonlinear term]\label{nonlinear_bound}
Let $\Omega \subset R^{d},$
$d=2,3.$ For all $u,v,w\in X $
\[
|b^{\ast}(u, v, w)| \le  C(\Omega)\sqrt{\| u \|\| \nabla u\|}\|
\nabla v\|\| \nabla w\|.
\]
\end{lemma}
\begin{proof}
See ~\cite{GR79}.
\end{proof}

We will also need the following inequalities: for any $u \in V$
\begin{eqnarray} \label{interp1}
\inf_{v \in V^h}\|\nabla (u - v)\| \le C(\Omega)\inf_{v \in
X^h}\|\nabla (u - v)\|,
\end{eqnarray}
\begin{eqnarray} \label{interp2}
\inf_{v \in V^h}\|u - v\| \le C(\Omega)\inf_{v \in X^h}\|\nabla (u -
v)\|,
\end{eqnarray}

The proof of (\ref{interp1}) can be found, e.g., in ~\cite{GR79},
and (\ref{interp2}) follows from the Poincare-Friedrich's inequality
and (\ref{interp1}).

We will also assume that the inverse inequality holds: there exists a constant $C$ independent of $h$, such that

\begin{equation}\label{inverseinequality}
||\nabla v|| \leq Ch^{-1}||v||, \mbox{   } \forall v \in X^h.
\end{equation}

Define also the number of time steps $N := \frac{T}{k}$.

We will use the error decomposition
\begin{equation}\label{errordecomposition}
\begin{split}
e_\ell^{i}=u^i-u_\ell^{h,i}=u^i-\tilde{u}^i+\tilde{u}^i-u_\ell^{h,i}=\eta_\ell^i-\phi_\ell^{h,i},\\
\mbox{where } \tilde{u}^i \in V^h \mbox{ is some projection of } u^i \mbox{ onto } V^h,\\
\mbox{and } \eta_\ell^i=u^i-\tilde{u}^i, \phi_\ell^{h,i}=u_\ell^{h,i}-\tilde{u}^i, \phi_\ell^{h,i} \in V^h, \forall i,\forall \ell=1,2.
\end{split}
\end{equation}

The  $L^2$ projection is defined in the usual way.
\begin{definition} \label{defpro}
The $L^2$ projection ${P}^H $of a given function $\mathbb{L}$ onto the finite element space $L^H$ is the solution of the following : find $\bar{\mathbb{L}}= {P}^H \mathbb{L}\in L^H$ such that
\begin{eqnarray} \label{pro}
(\mathbb{L}-{P}^H \mathbb{L}, S_H)=0,
\end{eqnarray}
for all $S_H \in L^H$.
\end{definition}
Hence, we get
\begin{eqnarray}\label{pro2}
    \|I-{P}^H \| &\leq& 1, \\
	\|(I-{P}^H) \mathbb{L}\| &\leq& CH^k\|\mathbb{L}\|_{k+1},
\end{eqnarray}
for all $\mathbb{L} \in (L(\Omega_i))^{d\times d}\cap (H^{k+1}(\Omega_i))^{d\times d}$.

We conclude the preliminaries by formulating the discrete Gronwall's
lemma, see, e.g. ~\cite{HR90}
\begin{lemma} \label{prelim02}
Let $k,B$, and $a_\mu, b_\mu, c_\mu, \gamma_\mu,$ for integers $\mu \ge 0$, be nonnegative numbers such that:
\[
a_n + k\sum_{\mu = 0}^{n}b_\mu \le k\sum_{\mu = 0}^{n}\gamma_\mu
a_\mu + k\sum_{\mu = 0}^{n}c_\mu + B \mbox{ for } n \ge 0.
\]
Suppose that $k\gamma_\mu < 1$ for all $\mu$, and set $\sigma_\mu =
(1 - k\gamma_\mu)^{-1}$. Then
\[
a_n + k\sum_{\mu = 0}^{n}b_\mu \le e^{k\sum_{\mu = 0}^{n}\sigma_\mu
\gamma_\mu}\cdot [k\sum_{\mu = 0}^{n}c_\mu + B].
\]
\end{lemma}

%%%%%%%%%%%%%%%%%%%%%%%%%%%%%%%%%%%%

\section{Stability and Error Estimates of the First Step Approximation}

The unconditional stability and error estimate of the first step approximation $u_1^h$ are presented in this section. Also using these results, an error estimate of its time derivative $\frac{e_1^{n+1}-e_1^{n}}{k}$ have been proved. 

Therefore, the formulation (\ref{VMSapp}) produces $O(h^m+H^{m}h+k)$ accurate, unconditionally stable approximation to the time-dependent Navier-Stokes equations.

We start by giving stability and error estimate of the modified Stokes Projection,
which we use as the approximation $\tilde u^0$ to the initial velocity $u_0$.

\begin{proposition} [Stability of the Stokes projection] \label{SSP}
Let $u$, $\tilde u$ satisfy (\ref{jan5}). The following bound holds
\begin{eqnarray} \label{jan6}
(\nu + h)\|\nabla \tilde u\|^2 \le 2(\nu + h)\|\nabla u\|^2 \\
\nonumber + 2d(\nu + h)^{-1}\inf_{q^h \in Q^h}\|p - q^h\|^2, \\
\nonumber \mbox{ where } d \mbox{ is the dimension, } d=2,3.
\end{eqnarray}
\end{proposition}

\begin{proposition}\label{ESP}{(Error estimate for Stokes Projection)}. Suppose the discrete inf-sup condition (\ref{inf_sup}) holds. Then the error in Stokes Projection satisfies

\begin{equation}
\begin{split}
(\nu + h)||\nabla(u-\tilde{u})||^2
\leq
C[(\nu + h)\inf_{v^h \in V^h}||\nabla(u-v^h)||^2\\
+(\nu + h)^{-1}\inf_{q^h \in Q^h}||p-q^h||^2],\\
\mbox{where C is a constant independent of $h$ and $\nu$}.
\end{split}
\end{equation}
\end{proposition}

\begin{proof}
Proofs can be found in \cite{L08}
\end{proof}

\begin{lemma} [Stability of the first step approximation] \label{SAV}
Let $u_1^{h,i}$ satisfy the equation (\ref{VMSapp}). Let $f \in
L^2(0,T;H^{-1}(\Omega))$. Then for $n=0,...,N-1$,

\begin{eqnarray}
\|u_1^{h,n+1}\|^2 + h\|\nabla  u_1^{h,n+1}\|^2 + \nu k\sum_{i=0}^{n+1}\|\nabla
u_1^{h,i}\|^2 \nonumber\\
+hk\sum_{i=0}^{n+1} \big(\|\nabla  u_1^{h,i+1} - \mathbb{G}_1^{\mathbb{H},i}\|^2
+ \|\nabla  u_1^{h,i} - \mathbb{G}_1^{\mathbb{H},i}\|^2\|\big) \nonumber \\
\le \|u^{s,0}\|^2 + h\|\nabla  u^{s,0}\|^2
+ \frac{1}{\nu}k\sum_{i=0}^{n+1}\|f(t_{i})\|_{-1}^2. \nonumber
\end{eqnarray}

\end{lemma}

\begin{proof}
Taking $v^h=u_1^{h,n+1} \in V^h$ in the equation (\ref{VMSapp}), and then applying Cauchy-Schwarz and Young's inequalities give:

\begin{equation}
\begin{split}
\frac{1}{2k}(||u_1^{h,n+1}||^2-||u_1^{h,n}||^2)+(\nu + h)||\nabla u_1^{h,n+1}||^2 - h(\mathbb{G}_1^{\mathbb{H},n},u_1^{h,n+1}) \\
\leq
(f(t_{n+1}),u_1^{h,n+1}). \label{avstab1}
\end{split}
\end{equation}

Also considering the fact that $(\nabla u_1^{h,n} - \mathbb{G}_1^{\mathbb{H},n}, \mathbb{G}_1^{\mathbb{H},n})=0,$ one can easily show
\begin{equation}
\|\nabla u_1^{h,n} - \mathbb{G}_1^{\mathbb{H},n}\|^2=
\|\nabla u_1^{h,n}\|^2-\|\mathbb{G}_1^{\mathbb{H},n}\|^2. \nonumber
\end{equation}
The last equality and some algebraic manipulations give
\begin{eqnarray}
(\nu + h)\|\nabla  u_1^{h,n+1}\|^2 - h(\mathbb{G}_1^{\mathbb{H},n},\nabla u_1^{h,n+1}) \nonumber\\
=
\nu\|\nabla  u_1^{h,n+1}\|^2 + \frac{h}{2}\big(\|\nabla  u_1^{h,n+1} - \mathbb{G}_1^{\mathbb{H},n}\|^2 + 2(\mathbb{G}_1^{\mathbb{H},n},\nabla u_1^{h,n+1}) - \|\mathbb{G}_1^{\mathbb{H},n}\|^2 \big) \nonumber \\
-h (\mathbb{G}_1^{\mathbb{H},n},\nabla u_1^{h,n+1}) +\frac{h}{2}\big(\|\nabla  u_1^{h,n+1}\|^2 - \|\nabla  u_1^{h,n}\|^2\big) +\frac{h}{2} \|\nabla  u_1^{h,n}\|^2 \nonumber \\
=
\nu\|\nabla  u_1^{h,n+1}\|^2 + \frac{h}{2}\|\nabla  u_1^{h,n+1} - \mathbb{G}_1^{\mathbb{H},n}\|^2
+\frac{h}{2} \|\nabla  u_1^{h,n} - \mathbb{G}_1^{\mathbb{H},n}\|^2 \nonumber \\
+\frac{h}{2}\big(\|\nabla  u_1^{h,n+1}\|^2
-\|\nabla  u_1^{h,n}\|^2\big).\label{avstab2}
\end{eqnarray}
\end{proof}

The definition of the dual norm with the regularity assumption on the forcing function followed by Cauchy-Schwarz and Young's inequalities produces

\begin{eqnarray}
(f(t_{n+1}),u_1^{h,n+1}) \leq \frac{1}{2\nu}\|f(t_{n+1})\|_{-1}^2 + \frac{\nu}{2}\|u_1^{h,n+1}\|^2.\label{avstab3}
\end{eqnarray}

Substituting \ref{avstab2} and \ref{avstab3} in \ref{avstab1}, we get

\begin{equation}
\begin{split}
\frac{1}{2k}(||u_1^{h,n+1}||^2-||u_1^{h,n}||^2)+\frac{\nu}{2}\|\nabla  u_1^{h,n+1}\|^2 \\
+ \frac{h}{2}\|\nabla  u_1^{h,n+1} - \mathbb{G}_1^{\mathbb{H},n}\|^2
+\frac{h}{2} \|\nabla  u_1^{h,n} - \mathbb{G}_1^{\mathbb{H},n}\|^2 \nonumber \\
+\frac{h}{2}\big(\|\nabla  u_1^{h,n+1}\|^2
-\|\nabla  u_1^{h,n}\|^2\big)
\leq
\frac{1}{2\nu}\|f(t_{n+1})\|_{-1}^2
\end{split}
\end{equation}

Multiplying both sides by $2k$ and summing over all time levels, the desired result can be found.

%%%%%%%%%%%

\begin{definition}
Let
$$C_u:=||u(x,t)||_{L^\infty(0,T;L^\infty(\Omega))},$$
$$C_{\nabla u}:=||\nabla u(x,t)||_{L^\infty(0,T;L^\infty(\Omega))},$$

and introduce $\tilde{C}$, satisfying
\begin{equation}\label{FEMbound}
\inf_{v \in V^h}||\nabla (u-v)|| \leq C_1\inf_{v \in X^h}||\nabla (u-v)|| \leq C_2h^m||u||_{H^{m+1}} \leq \tilde{C}h^m
\end{equation}

Also, using the constant $C(\Omega)$ from Lemma 2.3, we define
$\bar{C}:=1728C^4(\Omega)$.
\end{definition}

\begin{theorem}[Error estimate of the first step approximation]\label{EAV}
Let $f \in L^2(0,T;H^{-1})$, let $u_1^h$ satisfy (\ref{VMSapp}),
$$k\leq \frac{\nu + h}{18+4C_u^2+2(\nu + h)C_{\nabla u}+2\bar{C}\tilde{C}^4(\nu + h)^{-2}h^{4m}},$$
$$u \in L^2(0,T;H^{m+1}(\Omega)) \cap L^\infty(0,T;L^\infty(\Omega)), \nabla u \in L^\infty(0,T,L^\infty(\Omega)),$$
$$u_t \in L^2(0,T;H^{m+1}(\Omega)), u_{tt} \in L^2(0,T;L^2(\Omega)),p \in L^2(0,T;H^{m}(\Omega)).$$
Then there exist a constant $C=C(\Omega,T,u,p,f,\nu + h)$, such that
$$\max_{1\leq i \leq N}||u(t_i)-u_1^{h,i}||+\Big(  k\sum_{i=1}^{n+1} (\nu + h)||\nabla(u(t_i)-u_1^{h,i})||^2 \Big)^{1/2} \leq C(h^m+H^{m}h+k) $$
\end{theorem}

\begin{proof}
By Taylor expansion, $\frac{u(t_{n+1})-u(t_{n})}{k} = u_t(t_{n+1}) -
k\rho ^{n+1}$, where $\rho^{n+1} = u_{tt}(t_{n+\theta })$, for some $\theta \in [0,1]$. The variational formulation of the NSE, followed by the equations (\ref{VMSapp}), gives for $u \in X, p \in Q, u_1, u_2
\in X^h, p_1, p_2 \in Q^h, \forall v \in V^h$
\begin{eqnarray} \label{al116}
(\frac{u(t_{n+1})-u(t_{n})}{k}, v) + (\nu + h)(\nabla u(t_{n+1}), \nabla v) + b^{\ast}(u(t_{n+1}), u(t_{n+1}), v)\\
\nonumber - (p(t_{n+1}),\nabla \cdot v) = (f(t_{n+1}), v) + h(\nabla u(t_{n+1}),\nabla v) - k(\rho^{n+1}, v), \\
\label{al117}
(\frac{u_1^{h,n+1}-u_1^{h,n}}{k}, v) + (\nu + h)(\nabla u_1^{h,n+1}, \nabla v) + b^{\ast}(u_1^{h,n+1}, u_1^{h,n+1}, v)\\
\nonumber - (p_1^{h,n+1},\nabla \cdot v) = (f(t_{n+1}), v) + h(\mathbb{G}_1^{\mathbb{H},n},\nabla v_{h,1}).
\end{eqnarray}
Subtract (\ref{al117}) from (\ref{al116}). Introduce the error in
the AV approximation $e_1^i := u(t_i) - u_1^{h,i}, \forall i$. This
gives
\begin{eqnarray} \label{al119}
(\frac{e_1^{n+1}-e_1^n}{k}, v) + (\nu + h)(\nabla e_1^{n+1},\nabla v)\\
\nonumber + [b^{\ast}(u(t_{n+1}), u(t_{n+1}), v) - b^{\ast}(u_1^{h,n+1}, u_1^{h,n+1}, v)]\\
\nonumber - ((p(t_{n+1}) - p_1^{h,n+1}),\nabla \cdot v) = h(\nabla
u(t_{n+1}) - \mathbb{G}_1^{\mathbb{H},n},\nabla v) - k(\rho^{n+1}, v).
\end{eqnarray}
Adding and subtracting $b^{\ast}(u_1^{h,n+1}, u(t_{n+1}), v)$ to the
nonlinear terms in (\ref{al119}) gives
\begin{eqnarray} \label{december1}
b^{\ast}(u(t_{n+1}), u(t_{n+1}), v) - b^{\ast}(u_1^{h,n+1},
u_1^{h,n+1},
v)\\
\nonumber = b^{\ast}(e_1^{n+1}, u(t_{n+1}), v) +
b^{\ast}(u_1^{h,n+1}, e_1^{n+1}, v).
\end{eqnarray}

%Take $\tilde u^i \in V^h$ to be a projection of $u(t_{i})$ into
%$V^h, \forall i$.
Decompose the error
\begin{eqnarray} \label{december15}
e_1^i = u(t_{i}) - u_1^{h,i} = u(t_{i}) - \tilde u^i + \tilde u^i -
u_1^{h,i} =
\eta_1^i - \phi_1^{h,i}, \\
\nonumber \mbox{ where } \tilde u^i \in V^h \mbox{ is some
projection of } u(t_{i}) \mbox{ into } V^h, \\
\nonumber \mbox{ and } \eta_1^i = u(t_{i}) - \tilde u^i, \mbox{
 }\phi_1^{h,i} = u_1^{h,i} - \tilde u^i,
 \phi_1^{h,i}
\in V^h, \forall i.
\end{eqnarray}
Take $v = \phi_1^{h,n+1} \in V^h$ in (\ref{al119}) and use
(\ref{december1}). Using also $b^{\ast}(\cdot, \phi_1^{h,n+1},
\phi_1^{h,n+1}) = 0$ and $V^h \bot Q^h$, we obtain
\begin{eqnarray} \label{al120}
(\frac{\eta_1^{n+1} - \eta_1^n}{k}, \phi_1^{h,n+1}) - (\frac{\phi_1^{h,n+1} - \phi_1^{h,n}}{k}, \phi_1^{h,n+1})\\
\nonumber + (\nu + h)(\nabla \eta_1^{n+1}, \nabla \phi_1^{h,n+1}) - (\nu + h)\|\nabla \phi_1^{h,n+1}\|^2\\
\nonumber + b^{\ast}(\eta_1^{n+1}, u(t_{n+1}), \phi_1^{h,n+1}) - b^{\ast}(\phi_1^{h,n+1}, u(t_{n+1}), \phi_1^{h,n+1})\\
\nonumber + b^{\ast}(u_1^{h,n+1}, \eta_1^{n+1}, \phi_1^{h,n+1}) -
(p(t_{n+1}) - q^{h, n+1}, \nabla \cdot \phi_1^{h,n+1})\\
\nonumber = h(\nabla u(t_{n+1})- \mathbb{G}_1^{\mathbb{H},n}, \nabla \phi_1^{h,n+1}) -
k(\rho^{n+1}, \phi_1^{h,n+1}).
\end{eqnarray}

The equation (\ref{VMSapp}) states that $\mathbb{G}^{\mathbb{H},n}=P^H \nabla u^{h,n}$ where $P^H$ is the $L^2$-orthogonal projection defined by (\ref{pro}). Hence, utilizing Cauchy–Schwarz and Young's inequality, 
\begin{eqnarray}
\lefteqn{h(\mathbb{G}_1^{\mathbb{H},n}-\nabla u(t^{n+1}),\nabla \phi_1^{h,n+1})_{{\Omega_1}}}\nonumber\\
&\leq&(P^H \nabla (u_1^{h,n}-u(t^n)),\nabla \phi_1^{h,n+1})_{{\Omega_1}}-((I-P^H )\nabla u(t^n),\nabla \phi_1^{h,n+1})_{{\Omega_1}}\nonumber\\&&-(\nabla ( u(t^{n+1})-u(t^n)),\nabla \phi_1^{h,n+1})_{{\Omega_1}}\nonumber\\&\leq&\frac{h^2}{4\epsilon (\nu+h)}\Big(\|P^H \nabla\eta_1^n\|^2+\|P^H\nabla\phi_1^{h,n}\|^2\nonumber\\&&+\|(I-P^H )\nabla u(t^n)\|^2+\|\nabla ( u(t^{n+1})-u(t^n))\|^2\Big)\nonumber\\
&&+{\epsilon(\nu+h)\|\nabla \phi_1^{h,n+1}\|^2}. \label{er1}
\end{eqnarray}
Taylor remainder formula is used along with \eqref{pro}, \eqref{pro2} and inverse inequality to get
\begin{eqnarray}
\lefteqn{	h(\mathbb{G}_1^{\mathbb{H},n}-\nabla u(t^{n+1}),\nabla \phi_1^{h,n+1})_{{\Omega_1}}}\nonumber\\&\leq&\frac{h^2}{4\epsilon (\nu+h)}\Big(\|\nabla \eta_1^n\|^2+h^{-2}\|\phi_1^{h,n}\|^2+H^{2m}\|u(t^n)\|_{m+1}^2\nonumber\\&&+k^2\| u_t\|_{L^{\infty}(t^n,t^{n+1};H^1(\Omega))}^2\Big) +{\epsilon(\nu+h)\|\nabla \phi_1^{h,n+1}\|^2}.
\end{eqnarray}

Apply the Cauchy-Schwarz and Young's inequalities to (\ref{al120}).
Since $\|\nabla \cdot \phi_1^{h,n+1}\|^2 \le d\|\nabla
\phi_1^{h,n+1}\|^2$ for $\forall \epsilon
> 0$
\begin{eqnarray} \label{al121}
\frac{\|\phi_1^{h,n+1}\|^2 - \|\phi_1^{h,n}\|^2}{2k} + (\nu + h)\|\nabla \phi_1^{h,n+1}\|^2\\
\nonumber \le 6\epsilon (\nu + h)\|\nabla \phi_1^{h,n+1}\|^2 + \frac{1}{4\epsilon (\nu + h)}\|\frac{\eta_1^{n+1} - \eta_1^n}{k}\|_{-1}^2\\
\nonumber + |b^{\ast}(\eta_1^{n+1}, u(t_{n+1}), \phi_1^{h,n+1})| + |b^{\ast}(\phi_1^{h,n+1}, u(t_{n+1}), \phi_1^{h,n+1})|
\nonumber + |b^{\ast}(u_1^{h,n+1}, \eta_1^{n+1}, \phi_1^{h,n+1})|\\
\nonumber + \frac{h^2}{4\epsilon (\nu+h)}\Big(\|\nabla \eta_1^n\|^2+h^{-2}\|\phi_1^{h,n}\|^2+H^{2m}\|u(t^n)\|_{m+1}^2
+k^2\|u_t\|_{L^{\infty}(t^n,t^{n+1};H^1(\Omega))}^2\Big)\\
\nonumber + \frac{d}{4\epsilon (\nu + h)}\inf_{q^h \in Q^h}\|p(t_{n+1}) - q^{h, n+1}\|^2 + \frac{(\nu + h)}{4\epsilon}\|\nabla \eta_1^{n+1}\|^2
\frac{1}{4\epsilon (\nu + h)}k^2\|\rho^{n+1}\|_{-1}^2.
\end{eqnarray}
We bound the nonlinear terms on the right-hand side of
(\ref{al121}), starting now with the first one. Use the bound
(\ref{nonlin_usual}), the regularity of $u$ and Young's inequality
to obtain
\begin{eqnarray} \label{december2}
|b^{\ast}(\eta_1^{n+1}, u(t_{n+1}), \phi_1^{h,n+1})| \le \epsilon
(\nu + h)\|\nabla \phi_1^{h,n+1}\|^2 \\
\nonumber + C\frac{1}{\nu + h}\|\nabla \eta_1^{n+1}\|^2.
\end{eqnarray}
The second nonlinear term can be bounded, using the definition of
$b^{\ast}(\cdot, \cdot, \cdot)$ and the regularity of $u$. This
gives
\begin{eqnarray} \label{december3}
|b^{\ast}(\phi_1^{h,n+1}, u(t_{n+1}), \phi_1^{h,n+1})| \le
\frac{C_{\nabla u}}{2}\|\phi_1^{h,n+1}\|^2 +
\frac{C_{u}}{2}(|\phi_1^{h,n+1}|, |\nabla \phi_1^{h,n+1}|) \\
\nonumber \le \frac{C_{\nabla u}}{2}\|\phi_1^{h,n+1}\|^2 + \epsilon
(\nu + h)\|\nabla \phi_1^{h,n+1}\|^2 + \frac{C_u^2}{16 \epsilon
(\nu + h)}\|\phi_1^{h,n+1}\|^2.
\end{eqnarray}
For the third nonlinear term of (\ref{al121}), use the error
decomposition to obtain
\begin{eqnarray} \label{december4}
|b^{\ast}(u_1^{h,n+1}, \eta_1^{n+1}, \phi_1^{h,n+1})| \le
|b^{\ast}(u(t_{n+1}), \eta_1^{n+1}, \phi_1^{h,n+1})|\\
\nonumber + |b^{\ast}(\eta_1^{n+1}, \eta_1^{n+1}, \phi_1^{h,n+1})| +
|b^{\ast}(\phi_1^{h,n+1}, \eta_1^{n+1}, \phi_1^{h,n+1})|.
\end{eqnarray}
Use the regularity of $u$ and the inequality (\ref{nonlin_usual}) to
bound the first two terms on the right-hand side of
(\ref{december4}). Applying Lemma \ref{nonlinear_bound} to the third
term gives
\begin{eqnarray} \label{june7}
|b^{\ast}(\phi_1^{h,n+1}, \eta_1^{n+1}, \phi_1^{h,n+1})| \le
C(\om)\|\nabla
\phi_1^{h,n+1}\|^{3/2}\|\phi_1^{h,n+1}\|^{1/2}\|\eta_1^{n+1}\|.
\end{eqnarray}
We apply the Young's inequality to (\ref{june7}) with
$p=\frac{4}{3}$ and $q=4$. Finally it follows from (\ref{december4})
that
\begin{eqnarray} \label{december5}
|b^{\ast}(u_1^{h,n+1}, \eta_1^{n+1}, \phi_1^{h,n+1})| \le \epsilon
(\nu + h)\|\nabla \phi_1^{h,n+1}\|^2\\
\nonumber + \frac{C}{\nu + h}(\|\nabla \eta_1^{n+1}\|^2 + \|\nabla
\eta_1^{n+1}\|^4) \\
\nonumber + \frac{27C^4(\om)}{64\epsilon^3(\nu + h)^3}\|\nabla
\eta_1^{n+1}\|^4\|\phi_1^{h,n+1}\|^2, \\
\nonumber \mbox{  where } C(\om) \mbox{ is the constant from Lemma
\ref{nonlinear_bound} }.
\end{eqnarray}
Take $\epsilon = \frac{1}{18}$ in (\ref{al121}). Using the bounds
(\ref{december2})-(\ref{december5}), we obtain
\begin{eqnarray} \label{december6}
\frac{\|\phi_1^{h,n+1}\|^2 - \|\phi_1^{h,n}\|^2}{2k} + \frac{\nu + h}{2}\|\nabla \phi_1^{h,n+1}\|^2 \\
\nonumber \le \frac{C}{\nu + h}\|\frac{\eta_1^{n+1} - \eta_1^n}{k}\|_{-1}^2
\nonumber + C(\nu + h)\|\nabla \eta_1^{n+1}\|^2
\nonumber + \frac{C}{\nu + h}\inf_{q^h \in Q^h}\|p(t_{n+1}) - q^{h, n+1}\|^2\\
\nonumber + \frac{9h^2}{2(\nu+h)}\Big(\|\nabla \eta_1^n\|^2+h^{-2}\|\phi_1^{h,n}\|^2+H^{2m}\|u(t^n)\|_{m+1}^2
+k^2\|u_t\|_{L^{\infty}(t^n,t^{n+1};H^1(\Omega))}^2\Big)\\
\nonumber + \frac{C}{\nu + h}k^2\|\rho^{n+1}\|_{-1}^2 +
\frac{C}{\nu + h}(\|\nabla \eta_1^{n+1}\|^2 + \|\nabla
\eta_1^{n+1}\|^4)\\
\nonumber + (\frac{1}{2}C_{\nabla u} + \frac{C_u^2}{\nu + h} +
\frac{\bar{C}}{(\nu + h)^3}\|\nabla
\eta_1^{n+1}\|^4)\|\phi_1^{h,n+1}\|^2.
\end{eqnarray}

Sum (\ref{december6}) over all time levels and multiply by $2k$. It
follows from the regularity assumptions of the theorem that
\begin{eqnarray*}
k\sum_{i=0}^{n}\|\rho^{i+1}\|_{-1}^2 \le
Ck\sum_{i=0}^{n}\|\rho^{i+1}\|^2 \le C.
\end{eqnarray*}

Therefore we obtain
\begin{eqnarray} \label{al122}
\|\phi_1^{h,n+1}\|^2 + (\nu + h)k\sum_{i=0}^{n}\|\nabla \phi_1^{h,i+1}\|^2 \le (1+\frac{9k}{2(\nu+h)})\|\phi_1^{h,0}\|^2\\
\nonumber + \frac{2C}{\nu + h}k\sum_{i=0}^{n}[\|\frac{\eta_1^{i+1}-\eta_1^i}{k}\|_{-1}^2 + (\nu + h)^2\|\nabla \eta_1^i\|^2
+ \|\nabla \eta_1^i\|^2 \\
\nonumber + \|\nabla \eta_1^i\|^4 +\inf_{q^h \in Q^h}\|p(t_{i}) - q^{h, i}\|^2 + \|\phi_1^{h,i}\|^2 + h^2H^{2m}+ k^2]\\
+\frac{9}{(\nu + h)}k\sum_{i=0}^{n} \|\phi_1^{h,i+1}\|^2\\
\nonumber + k\sum_{i=0}^{n}(C_{\nabla u} + \frac{2C_u^2}{\nu + h}
+ \frac{2\bar{C}}{(\nu + h)^3}\|\nabla
\eta_1^{i+1}\|^4)\|\phi_1^{h,i+1}\|^2.
\end{eqnarray}

Take $\tilde u^i$ in the error decomposition (\ref{december15}) to
be the $L^2$-projection of $u(t_i)$ into $V^h$, for $i \ge 1$. Take
$\tilde u^0$ to be $u_0^s$. This gives $\phi_1^{h,0} = 0$ and $e_1^0
= \eta_1^0$. Also it follows from Proposition \ref{ESP}
that $\|\nabla \eta_1^0\| \le Ch^m$; under the assumptions of the
theorem the discrete Gronwall's lemma gives
\begin{eqnarray} \label{december7}
\|\phi_1^{h,n+1}\|^2 + (\nu + h)k\sum_{i=0}^{n}\|\nabla \phi_1^{h,i+1}\|^2\\
\nonumber \le \frac{C}{\nu + h}k\sum_{i=0}^{n}[\|\frac{\eta_1^{i+1}-\eta_1^i}{k}\|_{-1}^2 + \|\nabla \eta_1^i\|^2\\
\nonumber + \|\nabla \eta_1^i\|^4 + \inf_{q^h \in Q^h}\|p(t_{i}) -
q^{h, i}\|^2 + h^2H^{2m} + k^2].
\end{eqnarray}
Using the error decomposition and the triangle inequality, we obtain
\begin{eqnarray} \label{june71}
\|e_1^{n+1}\| \le \|\eta_1^{n+1}\| + \|\phi_1^{h,n+1}\|, \\
\nonumber \|e_1^{n+1}\|^2 \le 2\|\eta_1^{n+1}\|^2 +
2\|\phi_1^{h,n+1}\|^2, \\
\nonumber \|\nabla e_1^{i+1}\|^2 \le 2\|\nabla \eta_1^{i+1}\|^2 +
2\|\nabla \phi_1^{h,i+1}\|^2, \\
\nonumber k\sum_{i=0}^{n}(\nu + h)\|\nabla e_1^{i+1}\|^2 \\
\nonumber \le 2k\sum_{i=0}^{n}(\nu + h)\|\nabla \phi_1^{h,i+1}\|^2
+ 2k\sum_{i=0}^{n}(\nu + h)\|\nabla \eta_1^{i+1}\|^2.
\end{eqnarray}
Then it follows from (\ref{december7}),(\ref{june71}) that
\begin{eqnarray} \label{december8}
\|e_1^{n+1}\|^2 + k\sum_{i=0}^{n}(\nu + h)\|\nabla e_1^{i+1}\|^2\\
\nonumber \le \frac{C}{\nu + h}k\sum_{i=0}^{n}[\|\frac{\eta_1^{i+1}-\eta_1^i}{k}\|_{-1}^2 + \|\nabla \eta_1^i\|^2\\
\nonumber + \|\nabla \eta_1^i\|^4 + \inf_{q^h \in Q^h}\|p(t_{i}) -
q^{h, i}\|^2 + h^2H^{2m} + k^2].
\end{eqnarray}
Use the approximation properties of $X^h, Q^h$. Since the mesh nodes
do not depend upon the time level, it follows from
(\ref{interp1}),(\ref{interp2}) that
\begin{eqnarray} \label{june72}
k\sum_{i=0}^{n}\|\frac{\eta_1^{i+1}-\eta_1^i}{k}\|_{-1}^2 \le
Ck\sum_{i=0}^{n}\|\frac{\eta_1^{i+1}-\eta_1^i}{k}\|^2 \le
Ch^{2m}, \\
\nonumber k\sum_{i=0}^{n}\|\nabla \eta_1^i\|^2 \le Ch^{2m}, \\
\nonumber k\sum_{i=0}^{n}\inf_{q^h \in Q^h}\|p(t_{i}) - q^{h, i}\|^2
\le Ch^{2m}.
\end{eqnarray}

Hence, we obtain from (\ref{december8}),(\ref{june72}) that
\begin{eqnarray} \label{december9}
\|u(t_{n+1}) - u_1^{h,n+1}\|^2 + k\sum_{i=0}^{n}(\nu + h)\|\nabla (u(t_{n+1}) - u_1^{h,n+1})\|^2\\
\nonumber \le \frac{C}{\nu + h}[h^{2m} + h^2H^{2m} + k^2], \\
\nonumber \mbox{ where } C=C(\om,T,u,p,f).
\end{eqnarray}
This proves theorem.
\end{proof}

%%%%%%%%%%%%%%%%%%%%%%%%%%%%%%%%%%%%

The following lemma will be used in the proof of Theorem (\ref{dedt2}).
%%%%%%%%%%%%%%%%%%%%%%%%%%%%%%%%%%%%

\begin{lemma}\label{dedt1}
Let $f \in L^2(0,T;H^{-1}(\Omega))$. Suppose $\phi^{h,0}$ and $\phi^{h,1}$ to be the modified Stokes projections of the initial velocity and velocity at the first time level, respectively. Let $m \geq 2$ and $$k < \frac{4(\nu+h)}{13(4(\nu+h)C_{\nabla u}+3C_u^2)}.$$ 

Then there exist a constant $C=C(\Omega,T,u,p,f,\nu+h)$, such that

\begin{equation}
\begin{split}
||\frac{\phi^{h,1}-\phi^{h,0}}{k}||^2+\frac{13}{2}(\nu+h)k||\nabla \frac{\phi^{h,1}-\phi^{h,0}}{k}||^2
\leq C(kh^{2m}+h^2+k^2+k^2h^{2m-3}+H^{2m})
\end{split}
\end{equation}

\end{lemma}

\begin{proof}
From the Stokes Projection(\ref{Modified Stokes Projection}) and error decomposition(\ref{errordecomposition}), we have

\begin{equation}\label{eq35}
(\nu+h)(\nabla \phi^{h,0},\nabla v)- (\nu+h)(\nabla \eta^0,\nabla v)-(p^0-q,\nabla.v)=0
\end{equation}

On the other hand the solution at the first time level satisfies the following

\begin{equation}\label{eq36}
\begin{split}
||\frac{\phi^{h,1}-\phi^{h,0}}{k}||^2+(\nu+h)(\nabla \phi^{h,1},\nabla \frac{\phi^{h,1}-\phi^{h,0}}{k})+b^*(u(t_1),u(t_1),\frac{\phi^{h,1}-\phi^{h,0}}{k})\\
-b^*(u_1^{h,1},u_1^{h,1},\frac{\phi^{h,1}-\phi^{h,0}}{k}) + (p^1,\nabla.\frac{\phi^{h,1}-\phi^{h,0}}{k}) \\=h(\nabla u(t_1) - \mathbb{G}_1^{\mathbb{H},0},\nabla \frac{\phi^{h,1}-\phi^{h,0}}{k})+k(\rho^1,\frac{\phi^{h,1}-\phi^{h,0}}{k})\\
+(\frac{\eta^1-\eta^0}{k},\frac{\phi^{h,1}-\phi^{h,0}}{k})+(\nu+h)(\nabla \eta^1,\frac{\phi^{h,1}-\phi^{h,0}}{k}),\\
\mbox{where } k\rho^1=\frac{u(t_1)-u(t_0)}{k}-u_t^{1}=ku_{tt}^\theta \mbox{, for some } \theta \in  (0,k).
\end{split}
\end{equation}

Subtracting equation \ref{eq35} from equation \ref{eq36} for $v=\frac{\phi^{h,1}-\phi^{h,0}}{k}$, we have

\begin{equation}\label{eq37}
\begin{split}
||\frac{\phi^{h,1}-\phi^{h,0}}{k}||^2+k(\nu+h)||\nabla \frac{\phi^{h,1}-\phi^{h,0}}{k}||^2\\
+b^*(u(t_1),u(t_1),\frac{\phi^{h,1}-\phi^{h,0}}{k})-b^*(u_1^{h,1},u_1^{h,1},\frac{\phi^{h,1}-\phi^{h,0}}{k})
\\ -k(\frac{p^1-p^0}{k}-q,\nabla.\frac{\phi^{h,1}-\phi^{h,0}}{k})
\\=h(\nabla u(t_1) - \mathbb{G}_1^{\mathbb{H},0},\nabla \frac{\phi^{h,1}-\phi^{h,0}}{k})+(\rho^1,\frac{\phi^{h,1}-\phi^{h,0}}{k})+(\frac{\eta^1-\eta^0}{k},\frac{\phi^{h,1}-\phi^{h,0}}{k})\\
+k(\nu+h)(\nabla \frac{\eta^1-\eta^0}{k},\nabla \frac{\phi^{h,1}-\phi^{h,0}}{k})
\end{split}
\end{equation}

Adding and subtracting $b^*(u_1^{h,1},u(t_1),\frac{\phi^{h,1}-\phi^{h,0}}{k})$ to the nonlinear terms in equation (\ref{eq37})
together with error decomposition (\ref{errordecomposition}) gives

\begin{eqnarray}\label{eq38}
\lefteqn{b^*(u(t_1),u(t_1),\frac{\phi^{h,1}-\phi^{h,0}}{k})-b^*(u_1^{h,1},u_1^{h,1},\frac{\phi^{h,1}-\phi^{h,0}}{k})} \nonumber\\
&=&b^*(e_1^1,u(t_1),\frac{\phi^{h,1}-\phi^{h,0}}{k})+b^*(u_1^{h,1},e_1^1,\frac{\phi^{h,1}-\phi^{h,0}}{k}) \nonumber\\
&=& b^*(\phi^{h,1},u(t_1),\frac{\phi^{h,1}-\phi^{h,0}}{k})-b^*(\eta^1,u(t_1),\frac{\phi^{h,1}-\phi^{h,0}}{k}) \nonumber\\
&+&b^*(u_1^{h,1},\phi^{h,1},\frac{\phi^{h,1}-\phi^{h,0}}{k})-b^*(u_1^{h,1},\eta^1,\frac{\phi^{h,1}-\phi^{h,0}}{k})
\end{eqnarray}

Adding and subtracting $\phi^{h,0}$ to the first component of the first nonlinear term in the equation (\ref{eq38}) gives

\begin{equation}\label{eq39}
b^*(\phi^{h,1},u(t_1),\frac{\phi^{h,1}-\phi^{h,0}}{k})=kb^*(\frac{\phi^{h,1}-\phi^{h,0}}{k},u(t_1),\frac{\phi^{h,1}-\phi^{h,0}}{k})+b^*(\phi^{h,0},u(t_1),\frac{\phi^{h,1}-\phi^{h,0}}{k})
\end{equation}

In the first nonlinear term of (\ref{eq39}), applying Cauchy-Schwarz and Young's inequalities together with the regularity assumption of u and bound (\ref{nonlin_usual}) gives

\begin{equation}\label{in310}
\begin{split}
k|b^*(\frac{\phi^{h,1}-\phi^{h,0}}{k},u(t_1),\frac{\phi^{h,1}-\phi^{h,0}}{k})| \leq kC_{\nabla u}||\frac{\phi^{h,1}-\phi^{h,0}}{k}||^2\\+k\mu^*(\nu+h)||\nabla \frac{\phi^{h,1}-\phi^{h,0}}{k}||^2+k\frac{C_u^2}{16(\nu+h)\mu^*}||\frac{\phi^{h,1}-\phi^{h,0}}{k}||^2
\end{split}
\end{equation}

In the second nonlinear term of (\ref{eq39}), applying Cauchy Schwarz and Young's inequalities together with bound (\ref{nonlin_usual}) and inverse inequality (\ref{inverseinequality}) gives

\begin{equation}\label{in311}
|b^*(\phi^{h,0},u(t_1),\frac{\phi^{h,1}-\phi^{h,0}}{k})| \leq \mu||\frac{\phi^{h,1}-\phi^{h,0}}{k}||^2+\frac{Ch^{-2}}{4\mu}||\nabla \phi^{h,0}||^2
\end{equation}

In the second nonlinear term of (\ref{eq38}), applying Cauchy Schwarz and Young's inequalities together with bound (\ref{nonlin_usual}) and inverse inequality (\ref{inverseinequality}) gives

\begin{equation}\label{in312}
|b^*(\eta^1,u(t_1),\frac{\phi^{h,1}-\phi^{h,0}}{k})|\leq \mu||\frac{\phi^{h,1}-\phi^{h,0}}{k}||^2+\frac{Ch^{-2}}{4\mu}||\nabla \eta^1||^2
\end{equation}

For the third nonlinear term of equation (\ref{eq38}), applying error decomposition (\ref{errordecomposition}) gives

\begin{equation}\label{in313}
\begin{split}
|b^*(u_1^{h,1},\phi^{h,1},\frac{\phi^{h,1}-\phi^{h,0}}{k})|\leq |b^*(u(t_1),\phi^{h,1},\frac{\phi^{h,1}-\phi^{h,0}}{k})|+|b^*(\phi^{h,1},\phi^{h,1},\frac{\phi^{h,1}-\phi^{h,0}}{k})|\\
+|b^*(\eta^1,\phi^{h,1},\frac{\phi^{h,1}-\phi^{h,0}}{k})|
\end{split}
\end{equation}

Since nonlinear form is skew-symmetric in the second and third entry, we can replace terms like the first nonlinear term in the inequality (\ref{in313}) with terms like $|b^*(u(t_1),\phi^{h,0},\frac{\phi^{h,1}-\phi^{h,0}}{k})|$. Applying Cauchy-Schwarz and Young's inequalities together with the regularity assumption of u and inverse inequality gives

\begin{equation}\label{in314}
|b^*(u(t_1),\phi^{h,0},\frac{\phi^{h,1}-\phi^{h,0}}{k})| \leq 2\mu||\frac{\phi^{h,1}-\phi^{h,0}}{k}||^2+\frac{C_u^2}{4\mu}(||\nabla \phi^{h,0}||^2+h^{-2}||\phi^{h,0}||^2)
\end{equation}

Applying Young's inequality together with the Lemma (\ref{nonlinear_bound}) and inverse inequality (\ref{inverseinequality}) in the second nonlinear term of (\ref{in313}) gives

\begin{equation}\label{in315}
|b^*(\phi^{h,1},\phi^{h,1},\frac{\phi^{h,1}-\phi^{h,0}}{k})|=|b^*(\phi^{h,1},\phi^{h,0},\frac{\phi^{h,1}-\phi^{h,0}}{k})|
 \leq \mu||\frac{\phi^{h,1}-\phi^{h,0}}{k}||^2+\frac{Ch^{-3}}{4\mu}||\phi^{h,1}||^2||\nabla \phi^{h,0}||^2
\end{equation}

For the last nonlinear term in the inequality (\ref{in313}), we can apply (\ref{nonlin_usual}) and inverse inequality followed by Young's inequality to have

\begin{equation}\label{in316}
\begin{split}
|b^*(\eta^1,\phi^{h,1},\frac{\phi^{h,1}-\phi^{h,0}}{k})|=|b^*(\eta^1,\phi^{h,0},\frac{\phi^{h,1}-\phi^{h,0}}{k})|\\
\leq \mu||\frac{\phi^{h,1}-\phi^{h,0}}{k}||^2+\frac{Ch^{-2}}{4\mu}||\nabla \eta^1||^2||\nabla \phi^{h,0}||^2
\end{split}
\end{equation}

For the forth nonlinear term of equation (\ref{eq38}), applying error decomposition gives

\begin{equation}\label{in317}
\begin{split}
|b^*(u_1^{h,1},\eta^1,\frac{\phi^{h,1}-\phi^{h,0}}{k})|\leq |b^*(u(t_1),\eta^1,\frac{\phi^{h,1}-\phi^{h,0}}{k})|+|b^*(\phi^{h,1},\eta^1,\frac{\phi^{h,1}-\phi^{h,0}}{k})|\\
+|b^*(\eta^1,\eta^1,\frac{\phi^{h,1}-\phi^{h,0}}{k})|
\end{split}
\end{equation}

For all the nonlinear terms in the inequality (\ref{in317}), we can apply bound (\ref{nonlin_usual}) and inverse inequality followed by Young's inequality to have

\begin{equation}\label{in318}
|b^*(u(t_1),\eta^1,\frac{\phi^{h,1}-\phi^{h,0}}{k})| \leq \mu ||\frac{\phi^{h,1}-\phi^{h,0}}{k}||^2+\frac{Ch^{-2}}{4\mu}||\nabla \eta^1||^2
\end{equation}

\begin{equation}\label{in319}
|b^*(\phi^{h,1},\eta^1,\frac{\phi^{h,1}-\phi^{h,0}}{k})| \leq \mu||\frac{\phi^{h,1}-\phi^{h,0}}{k}||^2+Ch^{-4}||\nabla \eta^1||^2||\phi^{h,1}||^2
\end{equation}

\begin{equation}\label{in320}
\begin{split}
|b^*(\eta^1,\eta^1,\frac{\phi^{h,1}-\phi^{h,0}}{k})| \leq \mu||\frac{\phi^{h,1}-\phi^{h,0}}{k}||^2+Ch^{-2}||\nabla \eta^1||^4
\end{split}
\end{equation}

The equation (\ref{VMSapp}) states that $\mathbb{G}^{\mathbb{H},0}=P^H \nabla u(t_0)$ is the $L^2$-orthogonal projection of the initial value. Hence, utilizing Cauchy–Schwarz and Young's inequality, 
\begin{eqnarray}
\lefteqn{h(\nabla u(t_1)-\mathbb{G}_1^{\mathbb{H},0},\nabla \frac{\phi^{h,1}-\phi^{h,0}}{k})}\nonumber\\
&=& hk(\nabla \frac{u(t_1)-u(t_0)}{k},\nabla \frac{\phi^{h,1}-\phi^{h,0}}{k}) + h(\nabla u(t_0) - \mathbb{G}_1^{\mathbb{H},0},\nabla \frac{\phi^{h,1}-\phi^{h,0}}{k})\nonumber\\
&\leq& k\mu^*(\nu+h)||\nabla \frac{\phi^{h,1}-\phi^{h,0}}{k}||^2 + \frac{kh^2}{4\mu^*(\nu+h)}||\nabla \frac{u(t_1)-u(t_0)}{k}||^2\\
&+& C||(I-P^H) \nabla u(t_0)||^2 + \mu h^2||\nabla \frac{\phi^{h,1}-\phi^{h,0}}{k}||^2.\nonumber
\end{eqnarray}
Taylor remainder formula is used along with \eqref{pro}, \eqref{pro2} and inverse inequality to get
\begin{eqnarray}
\lefteqn{h(\nabla u(t_1)-\mathbb{G}_1^{\mathbb{H},0},\nabla \frac{\phi^{h,1}-\phi^{h,0}}{k})}\nonumber\\
&\leq& k\mu^*(\nu+h)||\nabla \frac{\phi^{h,1}-\phi^{h,0}}{k}||^2 + \frac{kh^2}{4\mu^*(\nu+h)}\| u_t\|_{L^{\infty}(t^0,t^1;H^1(\Omega))}^2\\
&+& CH^{2m}\|u(t_0)\|_{m+1}^2 + \mu ||\frac{\phi^{h,1}-\phi^{h,0}}{k}||^2.\nonumber
\end{eqnarray}

Apply Cauchy-Schwarz and Young's inequalities to (\ref{eq37}). Since $||\nabla.\frac{\phi^{h,1}-\phi^{h,0}}{k}|| \leq d||\nabla\frac{\phi^{h,1}-\phi^{h,0}}{k}||$,

\begin{equation}\label{in321}
\begin{split}
(1-12\mu-(\frac{C_{\nabla u}}{2}+\frac{C_u^2}{16(\nu+h)\mu^*})k)||\frac{\phi^{h,1}-\phi^{h,0}}{k}||^2\\
+(1-4\mu^*)(\nu+h)k||\nabla \frac{\phi^{h,1}-\phi^{h,0}}{k}||^2\\
\leq \frac{dk}{4\mu^*(\nu+h)}\inf_{q\in Q^h}||\frac{p^1-p^0}{k}-q||^2+\frac{k^2}{4\mu}||\rho^1||^2+\frac{1}{4\mu}||\frac{\eta^1-\eta^0}{k}||^2\\
+\frac{k(\nu+h)}{4\mu^*}||\nabla \frac{\eta^1-\eta^0}{k}||^2+\frac{Ch^{-2}}{4\mu}||\nabla \phi^{h,0}||^2+\frac{C_u^2}{4\mu}||\nabla \phi^{h,0}||^2+\frac{C_u^2h^{-2}}{4\mu}||\phi^{h,0}||^2\\
+\frac{Ch^{-3}}{4\mu}||\phi^{h,1}||^2||\nabla \phi^{h,0}||^2+\frac{Ch^{-2}}{4\mu}||\nabla \eta^1||^2||\nabla \phi^{h,0}||^2\\
+\frac{Ch^{-2}}{2\mu}||\nabla \eta^1||^2+Ch^{-4}||\phi^{h,1}||^2||\nabla \eta^1||^2+Ch^{-2}||\nabla \eta^1||^4\\
+ \frac{kh^2}{4\mu^*(\nu+h)}\| u_t\|_{L^{\infty}(t^0,t^1;H^1(\Omega))}^2
+ CH^{2m}\|u(t_0)\|_{m+1}^2.
\end{split}
\end{equation}

Use the approximation properties of $X^h, Q^h.$ Since the mesh nodes do not depend upon the time level, it follows from (\ref{interp1}), (\ref{interp2}) that

\begin{equation}\label{inek}
\begin{split}
\inf_{q \in Q}||\frac{p^{1}+p^{0}}{k}-q||^2 \leq Ch^{2m},\\
||\frac{\eta_2^{1}-\eta_2^{0}}{k}||^2
\leq
Ch^{2m+2},\\
||\eta_2^{1}||^2
\leq
Ch^{2m+2}.
\end{split}
\end{equation}

Taking $\mu=1/13$ and $\mu^*=1/8$ and using bounds (\ref{inek}) for each term, it follows from the regularity assumption of u that

\begin{equation}\label{in322}
\begin{split}
(\frac{1}{13}-(\frac{C_{\nabla u}}{2}+\frac{3C_u^2}{8(\nu+h)})k)||\frac{\phi^{h,1}-\phi^{h,0}}{k}||^2+\frac{1}{2}(\nu+h)k||\nabla \frac{\phi^{h,1}-\phi^{h,0}}{k}||^2\\
\leq C(h^{2m-2}+h^2+k^2+k^2h^{2m-3} + H^{2m})
\end{split}
\end{equation}
The last inequality implies the lemma statement.
\end{proof}

\begin{theorem}\label{dedt2}
Let the assumptions of Lemma (\ref{dedt1}) and Theorem (\ref{EAV}) be satisfied.

Let $k \leq \min \{
\frac{\nu+h}{2CC_{\nabla u}(\nu+h)+2CC_u^2},C(\nu+h)^{\frac{5}{3}},C(\nu+h)^3 \}$

Then
$$||\frac{e_1^{n+1}-e_1^n}{k}||^2 + k\sum_{i=0}^n(\nu+h)||\nabla \frac{e_1^{i+1}-e_1^i}{k}||^2 \leq C[h^{2m}+h^2+k^2]$$
\end{theorem}

\begin{proof}

Start with the proof of the bound for $\|\frac{\phi_1^{h,n+1} -
\phi_1^{h,n}}{k}\|$. Consider (\ref{al119}) with (\ref{december1})
for $n \ge 1$
\begin{eqnarray} \label{december17}
(\frac{e_1^{n+1}-e_1^n}{k}, v) + (\nu+h)(\nabla e_1^{n+1},\nabla v)\\
\nonumber + b^{\ast}(e_1^{n+1}, u(t_{n+1}), v) +
b^{\ast}(u_1^{h,n+1},
e_1^{n+1}, v)\\
\nonumber - ((p(t_{n+1}) - p_1^{h,n+1}),\nabla \cdot v) = h(\nabla
u(t_{n+1})-\mathbb{G}_1^{\mathbb{H},n},\nabla v) - k(\rho^{n+1}, v), \\
\nonumber \mbox{ where } k\rho^{n+1} = u_t(t_{n+1}) -
\frac{u(t_{n+1}) - u(t_{n})}{k}.
\end{eqnarray}
Take $v = \frac{\phi_1^{h,n+1}-\phi_1^{h,n}}{k} =: s^{h,n+1} \in
V^h$ in (\ref{december17}). Then consider (\ref{december17}) at the
previous time level and make exactly the same choice $v = s^{h,n+1}
\in V^h$. Subtract the equations, using the Taylor expansion to
simplify the last term on the right-hand side. We obtain
\begin{eqnarray} \label{december18}
k(\frac{\eta_1^{n+1}-2\eta_1^n + \eta_1^{n-1}}{k^2},s^{h,n+1}) - (s^{h,n+1} - s^{h,n}, s^{h,n+1})\\
\nonumber + (\nu+h)k(\nabla (\frac{\eta_1^{n+1}-\eta_1^n}{k}),\nabla s^{h,n+1}) - (\nu+h)k\|\nabla s^{h,n+1}\|^2\\
\nonumber + b^{\ast}(e_1^{n+1},
u(t_{n+1}), s^{h,n+1}) + b^{\ast}(u_1^{h,n+1}, e_1^{n+1}, s^{h,n+1})\\
\nonumber - b^{\ast}(e_1^{n},
u(t_{n}), s^{h,n+1}) - b^{\ast}(u_1^{h,n}, e_1^{n}, s^{h,n+1})\\
\nonumber - k(\frac{(p(t_{n+1}) - p_1^{h,n+1}) - (p(t_{n}) -
p_1^{h,n})}{k},\nabla \cdot s^{h,n+1})\\
\nonumber = hk(\nabla \frac{u(t_{n+1}) - u(t_{n})}{k} - \frac{\mathbb{G}_1^{\mathbb{H},n} - \mathbb{G}_1^{\mathbb{H},n-1}}{k},\nabla
s^{h,n+1}) - Ck^2(\rho_t^{n+1}, s^{h,n+1}), \\
\nonumber \mbox{ where } \rho_t^{n+1} = u_{ttt}(t_{n+\theta}) \mbox{
for some } \theta \in [0,1].
\end{eqnarray}
Consider the nonlinear terms of (\ref{december18}). Adding and
subtracting $b^{\ast}(e_1^{n}, u(t_{n+1}), s^{h,n+1})$ and
$b^{\ast}(u_1^{h,n+1}, e_1^{n}, s^{h,n+1})$ gives
\begin{eqnarray} \label{december19}
b^{\ast}(e_1^{n+1}, u(t_{n+1}), s^{h,n+1}) - b^{\ast}(e_1^{n},
u(t_{n}),
s^{h,n+1})\\
\nonumber + b^{\ast}(u_1^{h,n+1}, e_1^{n+1}, s^{h,n+1}) -
b^{\ast}(u_1^{h,n}, e_1^{n}, s^{h,n+1})\\
\nonumber = [b^{\ast}(e_1^{n+1}, u(t_{n+1}), s^{h,n+1}) -
b^{\ast}(e_1^{n}, u(t_{n+1}), s^{h,n+1})\\
\nonumber + b^{\ast}(e_1^{n}, u(t_{n+1}), s^{h,n+1}) -
b^{\ast}(e_1^{n},
u(t_{n}), s^{h,n+1})]\\
\nonumber + [b^{\ast}(u_1^{h,n+1}, e_1^{n+1}, s^{h,n+1}) -
b^{\ast}(u_1^{h,n+1}, e_1^{n}, s^{h,n+1})\\
\nonumber + b^{\ast}(u_1^{h,n+1}, e_1^{n}, s^{h,n+1}) -
b^{\ast}(u_1^{h,n}, e_1^{n}, s^{h,n+1})].
\end{eqnarray}
Use the error decomposition (\ref{december15}). Since
$b^{\ast}(\cdot, s^{h,n+1}, s^{h,n+1})=0$, it follows from
(\ref{december19}) that
\begin{eqnarray} \label{december20}
b^{\ast}(e_1^{n+1}, u(t_{n+1}), s^{h,n+1}) - b^{\ast}(e_1^{n},
u(t_{n}),
s^{h,n+1})\\
\nonumber + b^{\ast}(u_1^{h,n+1}, e_1^{n+1}, s^{h,n+1}) -
b^{\ast}(u_1^{h,n}, e_1^{n}, s^{h,n+1})\\
\nonumber = kb^{\ast}(\frac{\eta_1^{n+1} - \eta_1^{n}}{k},
u(t_{n+1}),
s^{h,n+1}) - kb^{\ast}(s^{h,n+1}, u(t_{n+1}), s^{h,n+1})\\
\nonumber + kb^{\ast}(e_1^{n+1}, \frac{u(t_{n+1}) - u(t_{n})}{k},
s^{h,n+1}) + kb^{\ast}(u_1^{h,n+1}, \frac{\eta_1^{n+1} -
\eta_1^{n}}{k},
s^{h,n+1})\\
\nonumber + kb^{\ast}(\frac{u_1^{h,n+1} - u_1^{h,n}}{k}, e_1^n,
s^{h,n+1}).
\end{eqnarray}
Use the regularity of $u$ and the Cauchy-Schwarz and Young's
inequalities to obtain the bounds on the terms in
(\ref{december20}). It follows from (\ref{nonlin_usual}) that for
any $\epsilon > 0$
\begin{eqnarray} \label{december21}
k|b^{\ast}(\frac{\eta_1^{n+1} - \eta_1^{n}}{k}, u(t_{n+1}), s^{h,n+1})|\\
\nonumber \le \epsilon (\nu+h)k\|\nabla s^{h,n+1}\|^2 +
\frac{C}{\nu+h}k\|\nabla (\frac{\eta_1^{n+1} -
\eta_1^{n}}{k})\|^2.
\end{eqnarray}
For the second term on the right-hand side of (\ref{december20}) use
the regularity of $u$ and the Cauchy-Schwarz and Young's
inequalities to obtain
%********************************************* JULY 20, 2007 *************************
%\begin{eqnarray} \label{december22}
%k|b^{\ast}(s^{h,n+1}, u(t_{n+1}), s^{h,n+1})| \\
%\nonumber \le \frac{1}{2}C_{\nabla u}k\|s^{h,n+1}\|^2 +
%\frac{1}{2}C_uk(|s^{h,n+1}|, |\nabla s^{h,n+1}|).
%\end{eqnarray}
%The Cauchy-Schwarz and Young's inequalities then give
%*************************************************************************************
\begin{eqnarray} \label{december23}
k|b^{\ast}(s^{h,n+1}, u(t_{n+1}), s^{h,n+1})| \le \epsilon
(\nu+h)k\|\nabla s^{h,n+1}\|^2 \\
\nonumber + \frac{C}{\nu+h}C_u^2k\|s^{h,n+1}\|^2 +
\frac{1}{2}C_{\nabla u}k\|s^{h,n+1}\|^2.
\end{eqnarray}
The third nonlinear term on the right-hand side of
(\ref{december20}) is bounded by
\begin{eqnarray} \label{december24}
k|b^{\ast}(e_1^{n+1}, \frac{u(t_{n+1}) - u(t_{n})}{k}, s^{h,n+1})|\\
\nonumber \le \epsilon (\nu+h)k\|\nabla s^{h,n+1}\|^2 +
\frac{C}{\nu+h}k\|\nabla e_1^{n+1}\|^2.
\end{eqnarray}
For the fourth nonlinear term, add and subtract $u(t_{n+1})$ to the
first term of the trilinear form. Using (\ref{nonlin_usual}) and
Lemma \ref{nonlinear_bound} leads to
%********************************************** JULY 20, 2007 ************************
%\begin{eqnarray} \label{december25}
%k|b^{\ast}(u_1^{h,n+1}, \frac{\eta_1^{n+1} - \eta_1^{n}}{k},
%s^{h,n+1})| \le k|b^{\ast}(u(t_{n+1}), \frac{\eta_1^{n+1} -
%\eta_1^{n}}{k}, s^{h,n+1})|\\
%\nonumber + k|b^{\ast}(e_1^{n+1}, \frac{\eta_1^{n+1} -
%\eta_1^{n}}{k}, s^{h,n+1})|.
%\end{eqnarray}
%Use (\ref{nonlin_usual}) for the first term on the right-hand side
%of (\ref{december25}) and Lemma \ref{nonlinear_bound} for the second
%term. We obtain
%*************************************************************************************
\begin{eqnarray} \label{december26}
k|b^{\ast}(u_1^{h,n+1}, \frac{\eta_1^{n+1} - \eta_1^{n}}{k},
s^{h,n+1})| \le 2\epsilon (\nu+h)k\|\nabla s^{h,n+1}\|^2\\
\nonumber + \frac{C}{\nu+h}k\|\nabla (\frac{\eta_1^{n+1} -
\eta_1^{n}}{k})\|^2 + Ck\|e_1^{n+1}\|\|\nabla e_1^{n+1}\|\|\nabla
(\frac{\eta_1^{n+1} - \eta_1^{n}}{k})\|^2.
\end{eqnarray}
For the fifth term add and subtract $u(t_{n+1})$ to the first term
of the trilinear form to obtain
\begin{eqnarray} \label{december27}
k|b^{\ast}(\frac{u_1^{h,n+1} - u_1^{h,n}}{k}, e_1^n, s^{h,n+1})| \le
k|b^{\ast}(\frac{u(t_{n+1}) - u(t_{n})}{k}, e_1^n, s^{h,n+1})|\\
\nonumber + k|b^{\ast}(\frac{\eta_1^{n+1} - \eta_1^{n}}{k}, e_1^n,
s^{h,n+1})| + k|b^{\ast}(s^{h,n+1}, e_1^n, s^{h,n+1})|.
\end{eqnarray}
Apply the result of Lemma \ref{nonlinear_bound} to the last
trilinear form in (\ref{december27}) and use the Young's inequality
with $p=\frac{4}{3}$ and $q=4$. This gives
\begin{eqnarray} \label{december28}
\lefteqn{k|b^{\ast}(\frac{u_1^{h,n+1} - u_1^{h,n}}{k}, e_1^n, s^{h,n+1})|} \nonumber\\
&\le& 3\epsilon (\nu+h)k\|\nabla s^{h,n+1}\|^2 +
\frac{C}{\nu+h}k\|\nabla e_1^n\|^2 \nonumber\\
&+& \frac{C}{\nu+h}k\|\nabla e_1^n\|^2\|\nabla
(\frac{\eta_1^{n+1} - \eta_1^{n}}{k})\|^2 +
\frac{C}{(\nu+h)^3}k\|\nabla e_1^n\|^4\|s^{h,n+1}\|^2. 
\end{eqnarray}

Applying Cauchy-Schwarz and Young’s inequalities 

\begin{eqnarray}
\lefteqn{hk(\nabla \frac{u(t_{n+1}) - u(t_{n})}{k} - \frac{\mathbb{G}_1^{\mathbb{H},n} - \mathbb{G}_1^{\mathbb{H},n-1}}{k},\nabla
s^{h,n+1})} \nonumber\\
&\leq& \frac{Ch^2 \cdot k}{\nu+h}\|\nabla \frac{u(t_{n+1}) - u(t_{n})}{k}\|^2 + \frac{Ch^2 \cdot k}{\nu+h}\|\frac{\mathbb{G}_1^{\mathbb{H},n} - \mathbb{G}_1^{\mathbb{H},n-1}}{k}\|^2 + \epsilon k(\nu+h)\|\nabla s^{h,n+1}\|^2
\end{eqnarray}

By the properties of the projection, error decomposition and the inverse inequality, the following can be found

\begin{eqnarray} \label{vms1}
\lefteqn{\|\frac{\mathbb{G}_1^{\mathbb{H},n} - \mathbb{G}_1^{\mathbb{H},n-1}}{k}\|^2 = \|P^H\nabla \frac{u_1^{h,n} - u_1^{h,n-1}}{k}\|^2} \nonumber\\
&\leq& \|\nabla \frac{u_1^{h,n} - u_1^{h,n-1}}{k}\|^2 \leq \|\nabla \frac{u(t_{n}) - u(t_{n-1})}{k}\|^2 + \|\nabla
(\frac{\eta_1^{n} - \eta_1^{n-1}}{k})\|^2 + h^{-2}\| s^{h,n}\|^2
\end{eqnarray}

Applying the Cauchy-Schwarz and Young's inequalities to
(\ref{december18}) and using the bounds
(\ref{december20})-(\ref{vms1}) give
\begin{eqnarray} \label{december29}
\frac{\|s^{h,n+1}\|^2 - \|s^{h,n}\|^2}{2} + (\nu+h)k\|\nabla
s^{h,n+1}\|^2 \\
\nonumber \le 13\epsilon (\nu+h)k\|\nabla s^{h,n+1}\|^2\\
\nonumber + \frac{C}{\nu+h}k\|\frac{\eta_1^{n+1} - 2\eta_1^{n} +
\eta_1^{n-1}}{k^2}\|_{-1}^2 + C(\nu+h)k\|\nabla
(\frac{\eta_1^{n+1} - \eta_1^{n}}{k})\|^2\\
\nonumber + \frac{C}{\nu+h}k \inf_{q^h \in
Q^h}\|\frac{p(t_{n+1}) - p(t_{n})}{k} - \frac{q^{h,n+1} - q^{h,n}}{k}\|^2\\
\nonumber + \frac{C}{\nu+h}k[\|\nabla (\frac{\eta_1^{n+1} -
\eta_1^{n}}{k})\|^2 + \|\nabla e_1^n\|^2 + \|\nabla
(\frac{\eta_1^{n+1} - \eta_1^{n}}{k})\|^2\|\nabla e_1^n\|^2]\\
\nonumber + Ck\|e_1^n\|^2\|\nabla e_1^n\|^2 + Ck\|\nabla
(\frac{\eta_1^{n+1} - \eta_1^{n}}{k})\|^4 + \frac{C}{\nu+h}k \cdot k^2\|\rho_t^{n+1}\|_{-1}^2\\
\nonumber + \frac{C}{\nu+h}k \cdot h^2\Big(\|\nabla(\frac{u(t_{n+1})-u(t_{n})}{k})\|^2 + \|\nabla \frac{u(t_{n}) - u(t_{n-1})}{k}\|^2 + \|\nabla
(\frac{\eta_1^{n} - \eta_1^{n-1}}{k})\|^2 \Big) \nonumber\\
+ \frac{C}{\nu+h}k \|s^{h,n}\|^2
+ C(C_{\nabla u} + \frac{C_u^2}{\nu+h} +
\frac{1}{(\nu+h)^3}\|\nabla e_1^n\|^4)k\|s^{h,n+1}\|^2.\nonumber
\end{eqnarray}

Since $u_{ttt} \in L^2(0,T;L^2(\Omega))$, we have
\begin{eqnarray*}
k\sum_{i=0}^n\|\rho_t^{i+1}\|_{-1}^2 \le
Ck\sum_{i=0}^n\|\rho_t^{i+1}\|^2 \le C.
\end{eqnarray*}

It follows from the assumption $k \le h$ and the result of Theorem
\ref{EAV} that
\begin{eqnarray*}
\max_{i}\|\nabla e_1^i\| \le C.
\end{eqnarray*}

Take $\epsilon = \frac{1}{26}$ in (\ref{december29}), simplify,
multiply both sides of the inequality by $2$ and sum over all time levels $n \ge 1$ to obtain
\begin{eqnarray} \label{december30}
\|s^{h,n+1}\|^2 + (\nu+h)k\sum_{i=1}^n\|\nabla s^{h,i+1}\|^2 \le
\|s^{h,1}\|^2\\
\nonumber + \frac{C}{\nu+h}k\sum_{i=1}^n[\|\frac{\eta_1^{i+1} -
2\eta_1^{i} + \eta_1^{i-1}}{k^2}\|_{-1}^2\\
\nonumber + (\nu+h)^2\|\nabla (\frac{\eta_1^{i+1} -
\eta_1^{i}}{k})\|^2 + \|\nabla (\frac{\eta_1^{i+1} -
\eta_1^{i}}{k})\|^2 \\
\nonumber + (\nu+h)\|\nabla (\frac{\eta_1^{i+1} - \eta_1^{i}}{k})\|^4\\
\nonumber + \inf_{q^h \in Q^h}\|\frac{p(t_{i+1}) - p(t_{i})}{k} -
\frac{q^{h,i+1} - q^{h,i}}{k}\|^2 + h^2 + k^2]\\
\nonumber + \frac{C}{(\nu+h)^2}k\sum_{i=1}^n(\nu+h)\|\nabla
e_1^i\|^2 + Ck\sum_{i=1}^n\|e_1^i\|^2 \\
\nonumber + \frac{C}{\nu+h} k\sum_{i=1}^n \|s^{h,i}\|^2\\
\nonumber + Ck\sum_{i=1}^n(C_{\nabla u} + \frac{C_u^2}{\nu+h} +
\frac{1}{(\nu+h)^3}\|\nabla e_1^i\|^4)\|s^{h,i+1}\|^2.
\end{eqnarray}

Since $\frac{C}{\nu+h} k\|s^{h,i+1}\|^2 \geq 0$, the following inequality holds. 

\begin{eqnarray}
\frac{C}{\nu+h} k\sum_{i=1}^n \|s^{h,i}\|^2 \leq \frac{C}{\nu+h} k\sum_{i=1}^n \|s^{h,i+1}\|^2 + \frac{C}{\nu+h} k\|s^{h,1}\|^2.
\end{eqnarray}

Substituting the last inequality in \ref{december30}, the following can be found.

\begin{eqnarray} \label{december31}
\|s^{h,n+1}\|^2 + (\nu+h)k\sum_{i=1}^n\|\nabla s^{h,i+1}\|^2 \le
(1 + \frac{C}{\nu+h} k)\|s^{h,1}\|^2\\
\nonumber + \frac{C}{\nu+h}k\sum_{i=1}^n[\|\frac{\eta_1^{i+1} -
2\eta_1^{i} + \eta_1^{i-1}}{k^2}\|_{-1}^2\\
\nonumber + (\nu+h)^2\|\nabla (\frac{\eta_1^{i+1} -
\eta_1^{i}}{k})\|^2 + \|\nabla (\frac{\eta_1^{i+1} -
\eta_1^{i}}{k})\|^2 \\
\nonumber + (\nu+h)\|\nabla (\frac{\eta_1^{i+1} - \eta_1^{i}}{k})\|^4\\
\nonumber + \inf_{q^h \in Q^h}\|\frac{p(t_{i+1}) - p(t_{i})}{k} -
\frac{q^{h,i+1} - q^{h,i}}{k}\|^2 + h^2 + k^2]\\
\nonumber + \frac{C}{(\nu+h)^2}k\sum_{i=1}^n(\nu+h)\|\nabla
e_1^i\|^2 + Ck\sum_{i=1}^n\|e_1^i\|^2 \\
\nonumber + Ck\sum_{i=1}^n(C_{\nabla u} + \frac{C_u^2 + 1}{\nu+h} +
\frac{1}{(\nu+h)^3}\|\nabla e_1^i\|^4)\|s^{h,i+1}\|^2.
\end{eqnarray}

Consider the error decomposition (\ref{december15}). Take $\tilde
u^i$ to be the $L^2$ projection of $u(t_i)$ into $V^h$, for all $i
\ge 1$. Since the mesh nodes do not depend upon the time level, it
follows from the approximation properties of $X^h,Q^h$ and the
regularity of $u,p$ that
\begin{eqnarray} \label{june94}
k\sum_{i=1}^n\|\frac{\eta_1^{i+1} - 2\eta_1^{i} +
\eta_1^{i-1}}{k^2}\|_{-1}^2 \le Ck\sum_{i=1}^n\|\frac{\eta_1^{i+1} -
2\eta_1^{i} + \eta_1^{i-1}}{k^2}\|^2 \le Ch^{2m}, \\
\nonumber k\sum_{i=1}^n\|\nabla (\frac{\eta_1^{i+1} -
\eta_1^{i}}{k})\|^2 \le Ch^{2m}, \\
\nonumber k\sum_{i=1}^n\|\nabla (\frac{\eta_1^{i+1} -
\eta_1^{i}}{k})\|^4 \le Ch^{4m}, \\
\nonumber k\sum_{i=1}^n\inf_{q^h \in Q^h}\|\frac{p(t_{i+1}) -
p(t_{i})}{k} - \frac{q^{h,i+1} - q^{h,i}}{k}\|^2 \le Ch^{2m}.
\end{eqnarray}

Using (\ref{june94}) and (\ref{december9}), we derive from
(\ref{december30}) that
\begin{eqnarray} \label{in323}
\|s^{h,n+1}\|^2 + (\nu+h)k\sum_{i=1}^n\|\nabla s^{h,i+1}\|^2 \le
(1 + \frac{C}{\nu+h} k)\|s^{h,1}\|^2\\
\nonumber + C[h^{2m} + h^2 + k^2]\\
\nonumber + Ck\sum_{i=1}^n(C_{\nabla u} + \frac{C_u^2+1}{\nu+h} +
\frac{1}{(\nu+h)^3}\|\nabla e_1^i\|^4)\|s^{h,i+1}\|^2.
\end{eqnarray}

In order to apply Gronwall's Lemma (\ref{prelim02}) in the inequality (\ref{in323}), we have to verify that \[Ck(C_{\nabla u}+\frac{C_u^2+1}{\nu+h}+\frac{1}{(\nu+h)^3}||\nabla e_1^i||^4) < 1.\]

To this end, we can first assume

\[Ck(C_{\nabla u}+\frac{C_u^2+1}{\nu+h}) < \frac{1}{2} \mbox{ and }
\frac{Ck}{(\nu+h)^3}||\nabla e_1^i||^4 < \frac{1}{2}.\]

Due to the first inequality, we have a bound on $k$ in the form
\[k < \frac{\nu+h}{CC_{\nabla u}(\nu+h)+C(C_u^2+1)}.\]

For the second inequality we investigate case by case.

\vspace{3mm}

For $k \leq h$,
it follows from the inverse inequality and theorem (\ref{EAV}) that

\begin{equation}\nonumber
\begin{split}
\frac{Ck}{(\nu+h)^3}||\nabla e_1^i||^4 \leq \frac{Ckh^{-4}}{(\nu+h)^3}||e_1^i||^4 \leq \frac{Ck}{(\nu+h)^3}(1+\frac{k}{h})^4\\
\leq \frac{Ck}{(\nu+h)^3} < \frac{1}{2}.
\end{split}
\end{equation}

Thus, we have a bound on $k$ in the form $k < C(\nu+h)^3$.

\vspace{3mm}

For $h \leq k$, it follows from the theorem (\ref{EAV}) that

$$\frac{Ck}{(\nu+h)^3}||\nabla e_1^i||^4 \leq \frac{Ck^{-1}}{(\nu+h)^5}(h^4+k^4) \leq \frac{2Ck^3}{(\nu+h)^5} < \frac{1}{2}.$$

\vspace{3mm}

It follows from the above calculations and theorem statement that

\[(C_{\nabla u}+\frac{C_u^2+1}{\nu+h}+\frac{1}{(\nu+h)^3}||\nabla e_1^i||^4)k < 1.\]

Now, we can apply discrete Gronwall's Lemma in the inequality (\ref{in323}) to have following bound

\begin{equation}\label{in324}
||\frac{\phi_1^{h,n+1}-\phi_1^{h,n}}{k}||^2+(\nu+h)k\sum_{i=1}^n||\nabla \frac{\phi_1^{h,i+1}-\phi_1^{h,i}}{k}||^2 \leq C[h^{2m}+h^2+k^2]
\end{equation}

Using the triangle inequality in the error decomposition (\ref{errordecomposition}), we obtain

\begin{equation}
||\frac{e_1^{n+1}-e_1^n}{k}||^2 + k\sum_{i=0}^n(\nu+h)||\nabla \frac{e_1^{n+1}-e_1^n}{k}||^2 \leq C[h^{2m}+h^2+k^2]
\end{equation}

This result proves the theorem.
\end{proof}

%%%%%%%%%%%%%%%%%%%%%%%%%%%%%%%%%%%%%%%%%%%%%%%%%%%%%%%%%%%%

\section{Stability and Error Estimate of Correction Step Approximation}

The correction step approximation presented here is identically same with that of the reference paper \cite{AL16}. Only differences are on their first step approximations. Therefore, the stability and accuracy analysis will be the same with the reference model up to the point when the first step approximation comes into play. For this reason, we are going to copy results from this paper up to some point, and then continue proving our theorem statements from there on.

Theoretical findings below illustrate that the formulation (\ref{CSapp}) produces $O(h^2+k^2)$ accurate, unconditionally stable correction step approximation to the time-dependent Navier-Stokes equations.

We first prove stability of the correction step approximation.

\begin{theorem} [Stability of the Correction Step Approximation]\label{SCS}

Let $f \in L^2(0,T;H^{-1}(\Omega))$, let $u_1^h,u_2^h$ satisfy (\ref{VMSapp}) and (\ref{CSapp}), respectively.
Then for n=0,...,N-1,

\begin{eqnarray}
\lefteqn{||u_2^{h,n+1}||^2+5h^2\nu^{-1}(\nu+h)^{-1}||u_1^{h,n+1}||^2} \nonumber\\
&+&5h^3\nu^{-1}(\nu+h)^{-1}k\sum_{i=0}^{n+1} \big(\|\nabla  u_1^{h,i+1} - \mathbb{G}_1^{\mathbb{H},i}\|^2
+ \|\nabla  u_1^{h,i} - \mathbb{G}_1^{\mathbb{H},i}\|^2\|\big) + k\sum_{i=1}^{n+1}(\nu + h)||\nabla u_2^{h,i}||^2 \nonumber\\
&\leq& C[||u_0^s||^2+(\nu + h)^{-1}k\sum_{i=1}^{n+1}||f(t_i)||_{-1}^2]. \nonumber
\end{eqnarray}

\end{theorem}

\begin{proof}
From the inequality (4.4) in \cite{AL16}, we have

\begin{equation}\label{in44}
\begin{split}
\frac{1}{2k}(||u_2^{h,n+1}||^2-||u_2^{h,n}||^2)+\frac{1}{2}(\nu + h)||\nabla u_2^{h,n+1}||^2 
\\
\leq \frac{5}{2(\nu+h)}||\frac{f(t_{n+1})-f(t_n)}{2}||_{-1}^2 
\\
+ \frac{5\nu^2k^2}{4(\nu + h)}C_{\nabla u_t}^2+\frac{5\nu^2k}{4(\nu+h)^2}k
(\nu + h)||\nabla(\frac{e_1^{n+1}-e_1^{n}}{k})||^2
\\
+\frac{5h^2}{2\nu(\nu + h)}\nu||\nabla u_1^{h,n+1}||^2
\\
+
\frac{5}{4\nu(\nu+h)^2}(\nu+h)k||\nabla(\frac{e_1^{h,n+1}-e_1^{h,n}}{k})||^2[\nu k||\nabla u_1^{h,n+1}||^2 + \nu k||\nabla u_1^{h,n}||^2]
\\
+\frac{5}{4\nu(\nu + h)}kC_{\nabla u_t}^2[\nu k||\nabla u_1^{h,n+1}||^2+\nu k||\nabla u_1^{h,n}||^2].
\end{split}
\end{equation}

Multiplying inequality by 2k and summing over all time levels followed by Lemma (\ref{SAV}) and Theorem (\ref{dedt2}) give

\begin{equation}
\begin{split}
||u_2^{h,n+1}||^2+k\sum_{i=1}^{n+1}(\nu + h)||\nabla u_2^{h,i}||^2
\\
\leq ||u_0^s||^2+ \frac{5}{(\nu + h)}k\sum_{i=1}^{n+1}||\frac{f(t_{i})-f(t_{i-1})}{2}||_{-1}^2 
\\
+\frac{5\nu^2k^3}{2(\nu + h)}C_{\nabla u_t}^2+\frac{5\nu^2k^2}{2(\nu + h)^2}C(h^{2m}+h^2+k^2)
\\
+\frac{5h^2}{\nu(\nu + h)}\Big[||u_0^s||^2+ h\|\nabla  u^{s,0}\|^2 -||u_1^{h,n+1}||^2\\
-hk\sum_{i=0}^{n+1} \big(\|\nabla  u_1^{h,i+1} - \mathbb{G}_1^{\mathbb{H},i}\|^2
+ \|\nabla  u_1^{h,i} - \mathbb{G}_1^{\mathbb{H},i}\|^2\|\big) +\frac{1}{\nu + h}k\sum_{i=1}^{n+1}||f(t_i)||_{-1}^2\Big]
\\
+\frac{5}{2\nu(\nu + h)}\big(\frac{h^{2m}+h^2+k^2}{\nu+h}+k^2C_{\nabla u_t}^2\big)
\Big[2||u_0^s||^2 +2h\|\nabla  u^{s,0}\|^2
\\
+\frac{1}{\nu + h}k\sum_{i=1}^{n+1}||f(t_i)||_{-1}^2+\frac{1}{\nu + h}k\sum_{i=1}^{n}||f(t_i)||_{-1}^2\Big].
\end{split}
\end{equation}

After some algebraic manipulation, we have the following inequality

\begin{equation}
\begin{split}
||u_2^{h,n+1}||^2+\frac{5h^2}{\nu(\nu + h)}||u_1^{h,n+1}||^2+\sum_{i=1}^{n+1}(\nu + h)||\nabla u_2^{h,i}||^2
\\
+\frac{5h^3}{\nu(\nu + h)}k\sum_{i=0}^{n+1} \big(\|\nabla  u_1^{h,i+1} - \mathbb{G}_1^{\mathbb{H},i}\|^2
+ \|\nabla  u_1^{h,i} - \mathbb{G}_1^{\mathbb{H},i}\|^2\|\big)
\\
\leq
||u_0^s||^2+ \frac{5}{(\nu+h)}k\sum_{i=1}^{n+1}||\frac{f(t_{i})-f(t_{i-1})}{2}||_{-1}^2 \\
+
\frac{5\nu^2k^3}{2(\nu + h)}C_{\nabla u_t}^2+\frac{5\nu^2k^2}{2(\nu + h)^2}C(h^{2m}+h^2+k^2)\\
+
C(||u_0^s||^2+h\|\nabla  u^{s,0}\|^2+\frac{1}{\nu + h}k\sum_{i=1}^{n+1}||f(t_i)||_{-1}^2).
\end{split}
\end{equation}
The last inequality implies the theorem statement.
\end{proof}

Theorem (\ref{SCS}) together with the Proposition (\ref{SSP}) proves the unconditional stability of both $u_1^{h,i}$ and $u_2^{h,i}$ for any $i \geq 0$.

The error estimate of the correction step approximation is given next.
%%%%%%%%%%%%%%%%%%%%%%%%%%%%%%%%%%%%

\begin{theorem}[Error Estimate of Correction Step Approximation]\label{ECS}
Let the assumptions of Theorem (\ref{dedt2}) be satisfied.
Let $$k < \frac{\nu + h}{(\nu + h)C_{\nabla u}+2C_u^2+(\nu + h)Ch^{m-1}+2Ch^{2m}}. $$
Then there exists a constant $C=C(\Omega,T,u,p,f,(\nu + h)^{-1}),$ such that

\begin{equation}\nonumber
\begin{split}
\max_{1\leq i \leq N}||u(t_i)-u_2^{h,i}|| + (k\sum_{i=0}^n(\nu + h)||\nabla(u(t_i)-u_2^{h,i}) ||^2)^{1/2}\\
\leq
C(h^{m}+h^2+k^2+hk).
\end{split}
\end{equation}

\end{theorem}

\begin{proof}
From the inequality (4.17) in \cite{AL16}, we have

\begin{equation}\label{in417}
\begin{split}
||\phi_2^{h,n+1}||^2+(\nu + h)k\sum_{i=0}^n||\nabla \phi_2^{h,i+1}||^2
\\
\leq \frac{C}{\nu + h}k\sum_{i=0}^n\Big[\inf_{q^h\in Q^h}||\frac{p^{h,i+1}+p^{h,i}}{2}-q^{h,i+1}||^2
\\
k^2||\nabla(\frac{e_1^{i+1}-e_1^{i+1}}{k})||^2
+h^2||\nabla e_1^{i+1}||^2+k^4
+||\frac{\eta_2^{i+1}-\eta_2^{i}}{k}||_{-1}^2
\\
+||\nabla \eta_2^{i+1}||^2
+k^2||\nabla e_1^{i+1}||^2
+||\nabla \eta_2^{i+1}||^4
\\
+k||\nabla(\frac{e_1^{i+1}-e_1^{i+1}}{k})||^2(k||\nabla e_1^{i+1}||^2+k||\nabla e_1^{i}||^2)
\\
+k\sum_{i=0}^n ||\phi_2^{h,i+1}||^2
\Big[
\frac{C_{\nabla u}}{2}+\frac{2C_{u}^2}{(\nu + h)}+\frac{1}{2}||\nabla \eta_2^{i+1}||
\\
+\frac{2}{\nu + h}||\nabla \eta_2^{i+1}||^2
\Big]+||\phi_2^{h,0}||^2
\end{split}
\end{equation}

Take $\tilde{u}^i$ in the error decomposition (\ref{errordecomposition}) to be the $L^2$-projection onto $V^h$, for $i \geq 1$. Take $\tilde{u}^0$ to be $u_0^s$. This gives $\phi_2^{h,0}=0$ and $e_1^0=\eta_2^0$. Also it follows from the Proposition (\ref{ESP}) that $||\eta_2^0|| \leq Ch^m$; under the assumption of the theorem applying the discrete Gronwall's lemma (\ref{prelim02}) and using bounds in theorems (\ref{EAV}), (\ref{dedt2}), give

\begin{equation}\label{in418}
\begin{split}
||\phi_2^{h,n+1}||^2+(\nu + h)k\sum_{i=0}^n||\nabla \phi_2^{h,i+1}||^2\\
\leq
\frac{C}{\nu + h}k\sum_{i=0}^n\Big[
\inf_{q^h\in Q^h}||\frac{p^{h,i+1}+p^{h,i}}{2}-q^{h,i+1}||^2
\\
+\frac{k^2}{\nu + h}(h^2+k^2)
+\frac{h^2}{\nu + h}(h^2+H^{2m}h^2+k^2)+k^4
\\
+||\frac{\eta_2^{i+1}-\eta_2^{i}}{k}||_{-1}^2
+||\nabla \eta_2^{i+1}||^2
+||\nabla \eta_2^{i+1}||^4
\\
+\frac{1}{(\nu + h)^2}(h^2+k^2)(h^2+H^{2m}h^2+k^2) \Big]+Ch^{2m}
\end{split}
\end{equation}

\vspace{3mm}

Use the approximation properties of $X^h, Q^h.$ Since the mesh nodes do not depend upon the time level, it follows from (\ref{interp1}), (\ref{interp2}) that

\begin{equation}\label{in419}
\begin{split}
k\sum_{i=0}^n \inf_{q^h\in Q^h}||\frac{p^{h,i+1}+p^{h,i}}{2}-q^{h,i+1}||^2 \leq Ch^{2m},\\
k\sum_{i=0}^n ||\frac{\eta_2^{i+1}-\eta_2^{i}}{k}||_{-1}^2 \leq
Ck\sum_{i=0}^n ||\frac{\eta_2^{i+1}-\eta_2^{i}}{k}||^2
\leq
Ch^{2m},\\
k\sum_{i=0}^n||\eta_2^{i+1}||^2
\leq
Ch^{2m}.
\end{split}
\end{equation}

Bounds (\ref{in418}) and (\ref{in419}) give the following result

\begin{equation}\label{in420}
\begin{split}
||\phi_2^{h,n+1}||^2+(\nu + h)k\sum_{i=0}^n||\nabla \phi_2^{h,i+1}||^2
\\
\leq
\frac{C}{(\nu+h)^2}(h^{2m}+H^{2m}(h^4+h^2k^2)+h^4+k^4+h^2k^2).
\end{split}
\end{equation}

Using the error decomposition and triangle inequality with (\ref{in420}), we obtain

\begin{equation}\label{in421}
\begin{split}
||e_2^{h,n+1}||+((\nu + h)k\sum_{i=0}^n||\nabla e_2^{h,i+1}||^2)^\frac{1}{2}\\
\leq
\frac{C}{(\nu + h)}(h^{m}+H^{m}(h^2+hk)+h^2+k^2+hk).
\end{split}
\end{equation}
\end{proof}
This proves the Theorem statement. Therefore, we derived the error estimates, which agree with the general theory of the defect and deferred correction methods. Clearly, the correction step approximation $u_2^h$ lifts the accuracy of an order of h in space and of k in time, compared to the first step approximation $u_1^h$.

Some computational results will be given next.

%%%%%%%%%%%%%%%%%%%%%%%%%%%%%%%%%%%%

\section{Computational Tests}

We perform one quantitative and one qualitative comparison test of SAV-DDC and AV-DDC models. Computational results with both tests not only support the theoretical findings of this paper but also illustrate superiority of SAV-DDC over AV-DDC.

Firstly, consider a manufactured true solution of NSE in $\Omega = [0,1]^2$ given by
\begin{eqnarray} \label{true}
u_1(x,y,t) &=& e^{-t}cos(2\pi(y-t)), \nonumber\\
u_2(x,y,t) &=& e^{-t}sin(2\pi(x-t)), \nonumber\\
p(x,y,t) &=& 0. \nonumber
\end{eqnarray}
The forcing function $f(x,y,t)$, the initial condition $u(x,y,0)$ and non-homogeneous boundary conditions are computed to comply with the given exact solution. Computations are ended at the final time $T=1$. The computations have been performed using the Taylor-Hood finite element space (P2/P1) for velocity and pressure pair, and also piecewise linear finite element space (P1) for the large scale space on the same mesh instead of piecewise quadratic finite element space (P2) on a different coarse mesh, see \cite{JK05}.

In particular, the exact solution is a rotational flow that moves along the line $y=x$ with a maximum velocity of 1 in each direction. Therefore, we choose the time step size as half of the mesh size, $\Delta t = h/2$; a possible analogue of the well-known CFL condition \cite{CFL23}. Also the additional viscosities in each case has been chosen equal to the time step size, and all these quantities have been refined together to observe convergence rates of the models.

The convergence rates in Tables \ref{table2}-\ref{table4} verify Theorems (\ref{EAV}) and (\ref{ECS}); the first step approximations produces first order of accuracy while the correction step approximation gives a second order of accuracy.

Comparing the first step approximations of each model, we observe that the convergence rates in the first step of AV-DDC has an asymptotic behaviour while that of SAV-DDC directly produces first order of accuracy with a better error estimate. On the other hand, defect-deferred correction methods rely mostly on the accuracy of the first step approximations. Therefore we can clearly conclude that employing SAV on the first step of defect-deferred correction methods contributes the overall accuracy of the correction step approximation. Also the computational results below show this expectation has been met.

For the first(i=1) and the correction(i=2) step approximations, define errors by:

\begin{eqnarray}
||e_i||_{L^2} = ||\,u_i-u^h\,||_{L^2(0,T;L^2(\Omega))}, \nonumber\\
||e_i||_{H^1} = ||\,u_i-u^h\,||_{L^2(0,T;H^1(\Omega))}. \nonumber
\end{eqnarray}

\begin{table}[H]
\caption{Errors and Convergence Rates(CR) with AV-DDC, $\nu=0.1$. \label{table1}}
\begin{tabular}{c|c|c|c|c||c|c|c|c|}
\cline{2-9}
                         & \multicolumn{4}{c||}{First Step} & \multicolumn{4}{c|}{Correction Step} \\ \hline
\multicolumn{1}{|l|}{$1/h$}  & $||e_1||_{L^2}$  & CR & $||e_1||_{H^1}$ & CR & $||e_2||_{L^2}$ & CR & $||e_2||_{H^1}$ & CR \\ \hline
\multicolumn{1}{|l|}{4}  & 0.155917         & -    & 1.44949         & -    & 0.0847247       & -    & 0.915718        & -    \\ \hline
\multicolumn{1}{|l|}{8}  & 0.103376         & 0.59 & 0.91002         & 0.67 & 0.0356702       & 1.25 & 0.346105        & 1.40 \\ \hline
\multicolumn{1}{|l|}{16} & 0.0618065        & 0.74 & 0.5425          & 0.75 & 0.0125766       & 1.50 & 0.118857        & 1.54 \\ \hline
\multicolumn{1}{|l|}{32} & 0.0341708        & 0.86 & 0.301312        & 0.85 & 0.0038449       & 1.71 & 0.0362162       & 1.71 \\ \hline
\end{tabular}
\end{table}

\begin{table}[H]
\caption{Errors and Convergence Rates(CR) with SAV-DDC, $\nu=0.1$. \label{table2}}
\begin{tabular}{c|c|c|c|c||c|c|c|c|}
\cline{2-9}
                         & \multicolumn{4}{c||}{First Step} & \multicolumn{4}{c|}{Correction Step} \\ \hline
\multicolumn{1}{|l|}{$1/h$}  & $||e_1||_{L^2}$  & CR & $||e_1||_{H^1}$ & CR & $||e_2||_{L^2}$ & CR & $||e_2||_{H^1}$ & CR \\ \hline
\multicolumn{1}{|l|}{4}  & 0.160372         & -    & 1.4644          & -    & 0.0899792       & -    & 0.948452        & -    \\ \hline
\multicolumn{1}{|l|}{8}  & 0.0701028        & 1.19 & 0.616771        & 1.25 & 0.0255807       & 1.81 & 0.262797        & 1.85 \\ \hline
\multicolumn{1}{|l|}{16} & 0.0306962        & 1.19 & 0.267739        & 1.20 & 0.00655849      & 1.96 & 0.0672375       & 1.97 \\ \hline
\multicolumn{1}{|l|}{32} & 0.0142972        & 1.10 & 0.124943        & 1.10 & 0.00166068      & 1.98 & 0.0170454       & 1.98 \\ \hline
\end{tabular}
\end{table}

\begin{table}[H]
\caption{Errors and Convergence Rates(CR) with AV-DDC, $\nu=0.01$. \label{table3}}
\begin{tabular}{c|c|c|c|c||c|c|c|c|}
\cline{2-9}
                         & \multicolumn{4}{c||}{First Step} & \multicolumn{4}{c|}{Correction Step} \\ \hline
\multicolumn{1}{|l|}{$1/h$}  & $||e_1||_{L^2}$  & CR & $||e_1||_{H^1}$ & CR & $||e_2||_{L^2}$ & CR & $||e_2||_{H^1}$ & CR \\ \hline
\multicolumn{1}{|l|}{4}  & 0.229077         & -    & 2.09789         & -    & 0.165639       & -    & 1.60463        & -    \\ \hline
\multicolumn{1}{|l|}{8}  & 0.175243         & 0.39 & 1.59616         & 0.39 & 0.105312       & 0.65 & 1.01741        & 0.66 \\ \hline
\multicolumn{1}{|l|}{16} & 0.118254         & 0.57 & 1.10838         & 0.53 & 0.0530586      & 0.99 & 0.557947       & 0.87 \\ \hline
\multicolumn{1}{|l|}{32} & 0.0714289        & 0.73 & 0.697369        & 0.67 & 0.0214237      & 1.31 & 0.261501       & 1.09 \\ \hline
\multicolumn{1}{|l|}{64} & 0.0399438        & 0.84 & 0.407656        & 0.77 & 0.00747879     & 1.52 & 0.105783       & 1.31 \\ \hline
\end{tabular}
\end{table}

\begin{table}[H]
\caption{Errors and Convergence Rates(CR) with SAV-DDC, $\nu=0.01$. \label{table4}}
\begin{tabular}{c|c|c|c|c||c|c|c|c|}
\cline{2-9}
                         & \multicolumn{4}{c||}{First Step} & \multicolumn{4}{c|}{Correction Step} \\ \hline
\multicolumn{1}{|l|}{$1/h$}  & $||e_1||_{L^2}$  & CR & $||e_1||_{H^1}$ & CR & $||e_2||_{L^2}$ & CR & $||e_2||_{H^1}$ & CR \\ \hline
\multicolumn{1}{|l|}{4}  & 0.304062         & -    & 2.6915          & -    & 0.252518       & -    & 2.29169        & -    \\ \hline
\multicolumn{1}{|l|}{8}  & 0.157858         & 0.94 & 1.47629         & 0.87 & 0.109739       & 1.20 & 1.0805         & 1.08 \\ \hline
\multicolumn{1}{|l|}{16} & 0.0743467        & 1.09 & 0.761911        & 0.95 & 0.0377188      & 1.54 & 0.453725       & 1.25 \\ \hline
\multicolumn{1}{|l|}{32} & 0.0353496        & 1.07 & 0.377853        & 1.01 & 0.0116789      & 1.69 & 0.166719       & 1.44 \\ \hline
\multicolumn{1}{|l|}{64} & 0.0171519        & 1.04 & 0.185847        & 1.02 & 0.00340097     & 1.78 & 0.054127       & 1.62 \\ \hline
\end{tabular}
\end{table}

%%%
For the qualitative testing, flow past a forward-backward facing step is considered. A $40\times10$ rectangular domain is used as the channel, and a $1\times1$ step is placed at the bottom of the channel, 5 units in. No-slip boundary conditions are strongly enforced on the walls of the channel and on the step, while parabolic inflow with maximum inlet 1 is introduced on the inflow boundary. Also on the outflow, 'do nothing' boundary condition is weakly enforced. The initial condition is set to be parabolic flow across the channel, and there is no external forcing, $f=0$. Viscosity $\nu=1/600$ is chosen in particular. For this setup the expected behavior is recirculating vortex formations behind the step and their detachment, see \cite{LMNR08},\cite{CHOR17},\cite{AKL19}. 

This comparison test is performed on the same coarse mesh (the smallest h=0.125) for both methods, and choose additional viscosity is equal to the time step size $\Delta t = 0.05$. Computations have been ended at the final time $T=40$. 

Figures \ref{fig1}-\ref{fig2} illustrate both method produces stable results. On the other hand, AV-DDC is too dissipative to capture vortex detachment, i.e. eddies which should detach and evolve remain attached and attain steady state, while SAV-DDC is able to reliably met with expectations of the problem setup and replicates the behavior of the flow given in the reference papers \cite{LMNR08},\cite{CHOR17},\cite{AKL19}. This test clearly shows that SAV-DDC is not over-dissipative as AV-DDC is, and hence, is able to capture turbulent characteristics of the flow better than AV-DDC.

%%%%%%%%%%%%%%%%%%%%%%%%%%%%%%%%%%%%%%%
\begin{figure}[H]
\centering     
\includegraphics[width=11.5cm,keepaspectratio]{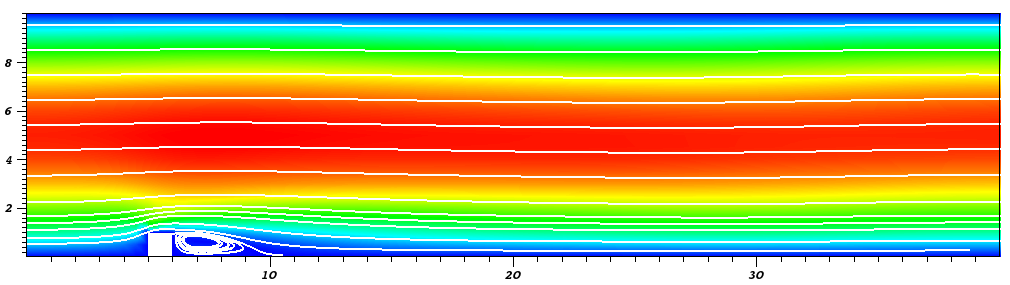}
\caption{AV-DDC}
\label{fig1}
\end{figure}
%%%%%%%%%%%%%%%%%%%%%%%%%%%%%%%%%%%%%%%%%%

%%%%%%%%%%%%%%%%%%%%%%%%%%%%%%%%%%%%%%%%%%
\begin{figure}[H]
\centering     
\includegraphics[width=11.5cm,keepaspectratio]{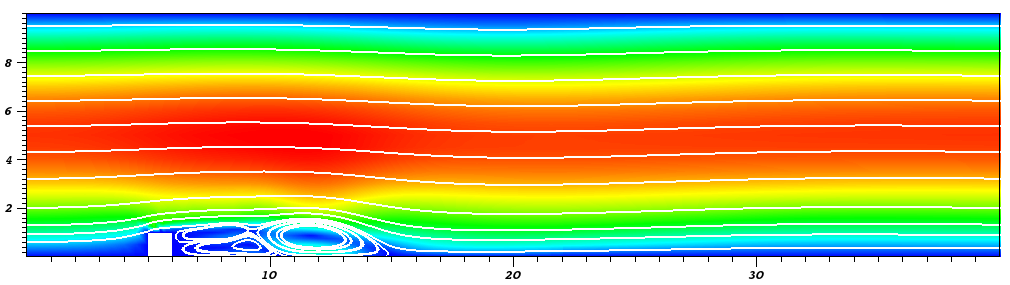}
\caption{SAV-DDC}
\label{fig2}
\end{figure}
%%%%%%%%%%%%%%%%%%%%%%%%%%%%%%%%%%%%%%%%%%

%%%%%%%%%%%%%%%%%%%%%%%%%%%%%%%%%%%%%%%%%%
Although the correction step approximations are computed with the same weak formulation, the first step approximation plays a great role in how accurate results they will give and how well the flow will be resolved.

\section{Conclusion}
The method presented here replaces the artificial viscosity approximation step of the defect-deferred correction method with an alternative to a projection-based subgrid artificial viscosity approximation. This alternative approach has both theoretically and computationally shown its superiority over conventional artificical viscosity approximation based defect-deferred correction method.

\end{document}